\documentclass{amsart}
\usepackage{latexsym,amsxtra,amscd,ifthen}
\usepackage{amsfonts}
\usepackage{verbatim}
\usepackage{amsmath}
\usepackage{amsthm}
\usepackage{amssymb}

\numberwithin{equation}{section}

\theoremstyle{plain}

\usepackage{color}

\newcommand\dashto{\mathrel{
  -\mkern-6mu{\to}\mkern-20mu{\color{white}\bullet}\mkern12mu
}}

\newtheorem{theorem}{Theorem}[section]
\newtheorem{lemma}[theorem]{Lemma}
\newtheorem{proposition}[theorem]{Proposition}
\newtheorem{corollary}[theorem]{Corollary}
\newtheorem{conjecture}[theorem]{Conjecture}
\newcommand{\inth}{\textstyle \int}
\theoremstyle{definition}
\newtheorem{definition}[theorem]{Definition}
\newtheorem{example}[theorem]{Example}
\newtheorem{convention}[theorem]{Convention}
\newtheorem{remark}[theorem]{Remark}
\newtheorem{question}[theorem]{Question}

\makeatletter              % This sequence of commands will
\let\c@equation\c@theorem  % incorporate equation numbering
\makeatother

\DeclareMathOperator{\GK}{GK}

\DeclareMathOperator{\Spec}{Spec}

\DeclareMathOperator{\Det}{Det}

\DeclareMathOperator{\Aut}{Aut}

\DeclareMathOperator{\GKdim}{GKdim}
\DeclareMathOperator{\End}{End} 

\DeclareMathOperator{\p}{{\sf p}}

\newcommand{\fm}{\mathfrak{m}}

\newcommand{\n}{\underline{n}}

\newcommand{\mop}{\operatorname{mop}}

\newcommand{\be}{\begin{enumerate}}
\newcommand{\ee}{\end{enumerate}}
\newcommand{\bq}{\begin{eqnarray*}}
\newcommand{\eq}{\end{eqnarray*}}
\newcommand{\bqn}{\begin{eqnarray}}
\newcommand{\eqn}{\end{eqnarray}}

\begin{document}

\title{Noncommutative Cyclic Isolated Singularities}

\author{Kenneth Chan, Alexander Young, James J. Zhang}

\address{(Chan) Department of Mathematics, Box 354350,
University of Washington, Seattle, Washington 98195, USA}

\email{kenhchan@math.washington.edu, ken.h.chan@gmail.com}

\address{(Young) Department of Mathematics,
DigiPen Institute of Technology, Redmond, WA 98052, USA}

\email{young.mathematics@gmail.com}

\address{(Zhang) Department of Mathematics, Box 354350,
University of Washington, Seattle, Washington 98195, USA}

\email{zhang@math.washington.edu}

\begin{abstract}
The question of whether a noncommutative graded quotient 
singularity $A^G$ is isolated depends on a subtle invariant 
of the $G$-action on $A$, called the pertinency. We prove a 
partial dichotomy theorem for isolatedness, which applies to 
a family of noncommutative quotient singularities arising 
from a graded cyclic action on the $(-1)$-skew polynomial 
ring. Our results generalize and extend some results of Bao, 
He and the third-named author and results of Gaddis, Kirkman, 
Moore and Won.
\end{abstract}


%16E65  (2000-now) Homological conditions on rings 
%(generalizations of regular, Gorenstein, Cohen-Macaulay rings, etc.)
%16S35  (1991-now) Twisted and skew group rings, crossed products
%16W22  (2000-now) Actions of groups and semigroups; invariant theory
%16S38  (2000-now) Rings arising from non-commutative algebraic geometry [See also 14A22]
%14J17  (1980-now) Singularities 


\maketitle

%\tableofcontents

\setcounter{section}{-1}
\section{Introduction}
\label{xxsec0}

Auslander \cite{Au} proved that if $G$ is a small finite subgroup of 
${\text{GL}}_n({\mathbb C})$, acting linearly on the symmetric algebra 
over ${\mathbb C}$ (namely, the commutative polynomial ring) 
$R:= {\mathbb C}[{\mathbb C}^{\oplus n}]$, with fixed subring $R^G$, 
then the natural map 
$$R\#  G\to \End_{R^G}(R)$$ 
is an isomorphism of graded algebras. Here $R\#  G$ denotes the skew 
group algebra associated to the $G$-action on $R$ and the hypothesis 
of $G$ being {\it small} means that $G$ does not contain any 
pseudo-reflections (e.g. $G$ is a finite subgroup of 
${\text{SL}}_n({\mathbb C})$). This theorem plays an important role 
in the McKay correspondence, relating representations of $G$ and those 
of $R^G$; and in the special case of dimension two, further relating 
configuration of the exceptional fibers in the minimal resolution of 
$\Spec R^G$. The noncommutative version of this theorem of Auslander 
is an important ingredient in establishing a noncommutative McKay 
correspondence, see \cite{CKWZ1, CKWZ2} for some recent developments. 
In \cite{BHZ1, BHZ2}, a numerical invariant was introduced for a 
semisimple Hopf algebra action on a (not necessarily commutative) 
algebra $R$ with finite Gelfand-Kirillov dimension (or GKdimension 
for short). The {\it pertinency} of a Hopf algebra $H$-action on $R$ 
\cite[Definition 0.1]{BHZ1} is defined to be
$$\p(R,H):=\GKdim (R)-\GKdim (R\# H/(e_0))$$
where $(e_0)$ is the two-sided ideal of the smash product $R\# H$ 
generated by the element $e_0:=1\# \inth$, where $\inth$ denotes an 
integral of $H$. One of the main results in \cite{BHZ1, BHZ2} is the 
following.

\begin{theorem} \cite[Theorem 0.3]{BHZ1}
\label{xxthm0.1}
Let $R$ be a noetherian, connected graded, Artin-Schelter regular, 
Cohen-Macaulay algebra of GKdimension at least 2. Let $H$ be a 
semisimple Hopf algebra acting on $R$ inner-faithfully and 
homogeneously. Then the following are equivalent:
\begin{enumerate}
\item[(1)]
$\p(R,H)\geq 2$.
\item[(2)]
The natural map $R\# H \to \End_{R^H}(R)$ is an isomorphism of 
graded algebras.
\end{enumerate}
\end{theorem}

The above theorem is useful for studying quotient singularities $R^H$ 
and for connecting the representation theory of $H$ and that of $R^H$. 
Several groups of researchers have computed the pertinency $\p(R,H)$ 
in different situations. A lower bound of the pertinency for the 
cyclic permutation action on the $(-1)$-skew polynomial rings and for 
the group actions on the universal enveloping algebra of some Lie 
algebras was given in \cite{BHZ1,BHZ2}; in \cite{GKMW}, the authors 
computed the pertinency for many new examples; the authors in \cite{HZ} 
introduced a new method of computing pertinency by using pertinent 
sequences; the paper \cite{CKZ} provided a lower bound of the pertinency 
for group coactions on noetherian graded down-up algebras.

Although many of these ideas can be applied to the Hopf algebra setting, 
in this paper we only consider group actions, namely, $H$ is a group 
algebra over a finite group $G$. When $G$ is acting on an algebra $R$, 
we usually assume that this action is inner-faithful. 

In algebraic geometry, singularities have been studied extensively. 
%{\footnote{I vote to drop the references \cite{LM, Mi}, or include 
%Miles Ried's paper Young Person's Guide to
%Canonical Singularities and some of Kollar's books.}} 
We recall the following basic result. When a small finite subgroup 
$G\subseteq {\text{GL}}_n({\mathbb C})$  acts naturally on the 
vector space $V:={\mathbb C}^{\oplus n}$, the quotient $V/G:
={\text{Spec}}({\mathbb C}[V]^{G})$ has isolated singularities if 
and only if $G$ acts freely on $V\setminus \{0\}$, see 
\cite[Lemma 2.1]{MSt}, \cite[Corollary to Lemma 2]{Fu} and 
\cite[p.7359]{MU1}.

In noncommutative algebraic geometry, Ueyama gave the following 
definition of a graded isolated singularity \cite[Definition 2.2]{Ue}. 
Let $B$ be a noetherian connected graded algebra. Then $B$ is a 
{\it graded isolated singularity} if the associated noncommutative 
projective scheme ${\text{tails}}\; B$ (in the sense of \cite{AZ}) 
has finite global dimension. Let $R$ be a noetherian Artin-Schelter 
regular algebra and $G$ a finite subgroup of the graded algebra 
automorphism group $\Aut_{gr}(R)$. Mori-Ueyama \cite[Theorem 3.10]{MU1} 
proved that if $\p(R,G)\geq 2$, then $R^G$ is a graded isolated 
singularity if and only if $\p(R,G)=\GKdim R$ (which is the largest 
possible). This result was extended to the Hopf algebra setting, 
namely, replacing $G$ by a semisimple Hopf algebra, in \cite{BHZ1}. 
The first few examples of graded isolated singularities in the 
noncommutative setting were given in 
\cite[Theorem 1.4, Examples 3.1, 4.7 and 5.5]{Ue} by mimicking
the commutative criterion of free action of $G$ on $V\setminus 
\{0\}$. More examples of graded isolated singularities were given 
in \cite{CKWZ1, CKWZ2, BHZ2, GKMW}. One example of graded isolated 
singularities in dimension three was given in \cite[Lemma 2.11(1)]{CKZ}. 
A more interesting example is \cite[Examples 5.4]{Ue} or 
\cite[Example 3.1]{KKZ1}, where the $G$-action on the degree one 
piece of the regular algebra $R$ is not free. We call such a graded 
isolated singularity {\it non-conventional} [Definition \ref{xxdef10.1}].

Since Mori-Ueyama's condition of maximal pertinency is not easy to 
check in general, we only obtain some special examples of graded 
isolated singularities in high GKdimension \cite{BHZ2}. It would be 
nice to understand exactly when the pertinency is maximal, but it 
seems extremely difficult to achieve this goal. The main object of 
this paper is to calculate a family of pertinencies all together, 
using induction. As a consequence, we obtain new examples of graded 
isolated singularities in arbitrarily large GKdimension. 

We now fix some notation. Let $\Bbbk$ be an algebraically closed 
field of characteristic zero. Let $n$ be an integer $\geq 2$. The 
algebra that we are interested in is the $(-1)$-skew polynomial ring 
$$\Bbbk_{-1}[\mathbf{x}]:=\Bbbk_{-1}[x_0,\ldots,x_{n-1}]$$ 
that is generated by $\{x_0,\ldots,x_{n-1}\}$ and subject to the 
relations
$$x_i x_j=(-1)x_j x_i$$
for all $i\neq j$. Let $C_n$ be the cyclic group of order $n$
acting on $\Bbbk_{-1}[\mathbf{x}]$ by permuting the generators
of the algebra cyclically, namely, $C_n$ is generated by
$\sigma=(012\cdots n-1)$ of order $n$ that acts on the generators by
$$\sigma \ast x_i= x_{i+1}$$
for all $i\in {\mathbb Z}_n:={\mathbb Z}/n {\mathbb Z}$. 
We have two results which establish a partial dichotomy.

\begin{theorem}
\label{xxthm0.2} 
Let $A:=\Bbbk_{-1}[\mathbf{x}]$ and $G:=C_n$. If $n=2^a p^b$ for 
some prime $p\geq 7$ and integers $a,b\geq 0$, then $\p(A, G)=
\GKdim (A)=n$. As a consequence, $A^G$ is a graded isolated 
singularity.
\end{theorem}

\begin{remark}
\label{xxrem0.3}
\begin{enumerate}
\item[(1)]
Theorem \ref{xxthm0.2} is a generalization of \cite[Examples 5.4]{Ue}
(when $n=2$) and \cite[Theorem 5.7(4)]{BHZ1} (when $n=2^a$ for 
some $a\geq 1$).
\item[(2)]
Although \cite[Examples 5.4]{Ue} and \cite[Theorem 5.7(4)]{BHZ1} have 
already provided examples of non-conventional graded isolated 
singularities of a similar type, Theorem \ref{xxthm0.2} is still 
quite surprising and counter-intuitive. 

Note that $\sigma\mid_{V}$ (where $V=\oplus_{i=0}^{n-1} \Bbbk x_i$) 
has eigenvalues $\{1, \xi,\xi^2,\ldots,\xi^{n-1}\}$ where $\xi$ is 
a primitive  $n$th root of unity. In particular, there is an 
eigenvalue of $\sigma$ on $V$ that is $1$ (which is not a primitive 
$n$th root of unity) with eigenvector $\sum_{i=0}^{n-1}x_i$ in $V$, 
or equivalently, the isolated singularity is non-conventional.
\item[(3)]
In fact, almost all graded isolated singularities considered in this 
paper will be non-conventional. One aim of this paper is to show that 
non-conventional graded isolated singularities are common in the 
noncommutative setting.
\item[(4)]
The proof of Theorem \ref{xxthm0.2} is very involved, using several
steps of reduction and induction. We hope to have a more conceptual 
proof in the future.
\end{enumerate} 
\end{remark}

When $p=3$ or $5$, Theorem \ref{xxthm0.2} fails.

\begin{theorem}
\label{xxthm0.4} 
Let $A:=\Bbbk_{-1}[\mathbf{x}]$ and $G:=C_n$. If either $3$ or $5$ 
divides $n$, then $\p(A, G)<\GKdim (A)=n$. Consequently, $A^G$ is 
not a  graded isolated singularity.
\end{theorem}

Combining the above two theorems, if $n=2^ap^b$ for some prime 
number $p$, then $A^{C_n}$ is a graded isolated singularity if 
and only if $p\neq 3,5$. It is not obvious to us why the primes 3 
and 5 are different from other primes in this situation. Based on 
the above two results we make a conjecture.

\begin{conjecture}
\label{xxcon0.5} 
Let $A:=\Bbbk_{-1}[\mathbf{x}]$ and $G:=C_n$. Then $A^G$ is 
a graded isolated singularity if and only if $n$ is not 
divisible by $3$ and $5$. 
\end{conjecture}

The above conjecture holds for $n$ less than $77$ following
Theorems \ref{xxthm0.2} and \ref{xxthm0.4}.

\begin{corollary}
\label{xxcor0.6} 
If $n<77$, then Conjecture \ref{xxcon0.5} holds.
\end{corollary}

Theorem \ref{xxthm8.7} provides further evidence for 
Conjecture \ref{xxcon0.5}. For general $n$ we have the 
following lower bound. Let
 
\begin{equation}
\label{E0.6.1}\tag{E0.6.1}
\phi_{2}(n)=\{k \mid 0\leq k\leq n-1\;{\text{with}}\;
\gcd(k,n)=2^w \;\; {\text{for some $w\geq 0$}}\}.
\end{equation}

\begin{theorem}
\label{xxthm0.7} 
Let $A:=\Bbbk_{-1}[\mathbf{x}]$ and $G:=C_n$. Then 
$\p(A, G)\geq |\phi_2(n)|$. 
\end{theorem}

Note that Theorem \ref{xxthm0.7} is an improvement of 
\cite[Theorem 5.7]{BHZ1} when $n$ is even. Combining 
Theorems \ref{xxthm0.2}, \ref{xxthm0.4}, \ref{xxthm0.7}
and further analysis, we have the following table of 
pertinencies.

\begin{proposition}
\label{xxpro0.8}
Let $p=\p(A,C_n)$. Then 

\[ \begin{array}{|c|l|l|l|l|l|l|l|l|l|l|l|l|l|} 
\hline n & 2 & 3 & 4& 5& 6& 7&8&9&10&11&12&13&14\\ 
\hline p &2 & 2 & 4& 4& 5& 7&8&8& 9 &11&\in [8,11]&13 &14\\ \hline 
\end{array} \]
where the notation $\in [8,11]$ means that $8\leq p\leq 11$.
\end{proposition}

By Proposition \ref{xxpro0.8}, the integer 12 is the smallest 
$n$ where that the exact value of $\p(A,C_n)$ is unknown. It 
would be nice to have exact values of $\p(A,C_n)$ for all $n$. 
In particular, we ask:

\begin{question}
\label{xxque0.9} Retain the above notation.
\begin{enumerate}
\item[(1)]
If $n$ is divisible by either $3$ or $5$, what
is the exact value of $\p(A,C_n)$? 
\item[(2)]
Does the sequence, from Proposition \ref{xxpro0.8},
\begin{equation}
\label{E0.9.1}\tag{E0.9.1}
2,2,4,4,5,7,8,8,9,11... 
\end{equation}
match up with any other sequences in literature? 
The On-Line Encyclopedia of Integer Sequences
website

https://oeis.org/

\noindent
does not give any sequences that match up with 
\eqref{E0.9.1}
\end{enumerate}
\end{question}

Graded isolated singularities have various special properties. 
Ueyama and Mori-Ueyama investigated certain properties of 
graded isolated singularities from the viewpoint of derived 
categories and representation theory. As an immediate 
consequence of \cite{Ue, MU1, MU2}, we have the following. 
We refer to \cite{Ue, MU1, MU2} for undefined terms in the 
next corollary.

\begin{corollary}
\label{xxcor0.10} 
Suppose $n=2^a p^b$ for some prime $p\geq 7$ and 
integers $a,b\geq 0$. Then the following hold.
\begin{enumerate}
\item[(1)]
${\text{tails}} \; A^G\cong {\text{tails}} \; A\#  G$.
\item[(2)]
$A$ is a $(n-1)$-cluster tilting object
in the category of graded maximal Cohen-Macaulay modules 
over $A^G$.
\item[(3)]
The derived category $D^b({\text{tails}} \; A^G)$ has a tilting 
object.
\item[(4)]
The derived category $D^b({\text{tails}} \; A^G)$  has a Serre 
functor.
\end{enumerate}
\end{corollary}

This paper is organized as follows. We provide background material 
in Section \ref{xxsec1}. Theorem \ref{xxthm0.7} is proven in 
Section \ref{xxsec2}. In Section \ref{xxsec3}, we give some 
preliminary results and Theorem \ref{xxthm0.4} is proven in Section 
\ref{xxsec4}. We continue some preparation in Sections \ref{xxsec5} 
and \ref{xxsec6}. The main result, Theorem \ref{xxthm0.2}, is 
proven in Section \ref{xxsec7}. In Section \ref{xxsec8}, we discuss
some partial results when $n=p_1 p_2$. Proposition \ref{xxpro0.8} 
is proven in Section \ref{xxsec9}. In Section \ref{xxsec10}, we 
construct more examples of non-conventional graded isolated 
singularities. The final section contains some questions and 
comments.

\subsection*{Acknowledgments}
%The authors thank the referee for his/her very careful reading and 
%valuable comments.
The authors thank Jason Bell, Ken Goodearl, Zheng Hua, Lance 
Small, Agata Smoktunowicz and Robert Won for many useful 
conversations on the subject and thank Jason Bell for the 
proof of Lemma \ref{xxlem10.6}. J.J. Zhang was partially 
supported by the US National Science Foundation
(No. DMS-1700825).

\section{Preliminaries}
\label{xxsec1}

Throughout let $\Bbbk$ be a base field that is algebraically closed 
of characteristic zero. All objects are $\Bbbk$-linear. 

An algebra $R$ is called {\it connected graded} if 
$R=\bigoplus_{n\geq 0} R_n$ satisfying $R_i R_j\subseteq R_{i+j}$ 
for all $i,j$ and $1\in R_0=\Bbbk$. We say $R$ is {\it locally 
finite} if $\dim_{\Bbbk} R_n<\infty$ for all $n$. In this paper 
all connected graded algebras will be locally finite.

We refer to \cite{KL, MR} for the definition of the 
{\it Gelfand-Kirillov dimension} (or {\it GKdimension}) of an 
algebra or a module. When $R$ is connected graded and 
finitely generated, its GKdimension is equal to 
\begin{equation}
\label{E1.0.1}\tag{E1.0.1}
\GKdim (R)=\limsup_{n\to\infty} \log_{n}(\sum_{i=0}^{n} 
\dim_{\Bbbk} R_i).
\end{equation}

Observe that $\GKdim(R) = 0$ if and only if $\dim_{\Bbbk} R 
< \infty$. For $q \in \Bbbk^\times$, the $q$-polynomial ring 
$$\Bbbk_q[x_1, \dots, x_m]:= 
\Bbbk \langle x_1, \ldots, x_m\rangle/
(x_i x_j - qx_j x_i\mid i<j)$$
has GKdimension $m$ (equal to the number of generators). 
If $B$ is either a subalgebra or a homomorphic image 
of an algebra $R$, then $\GKdim(B)\leq\GKdim(R)$.

Let $B$ be a noetherian connected graded algebra. If $M$ is a 
finitely generated graded right $B$-module, then we have a 
formula similar to \eqref{E1.0.1}, see \cite[p.1594]{SZ},
\begin{equation}
\label{E1.0.2}\tag{E1.0.2}
\GKdim (M)=\limsup_{n\to\infty} \log_{n}(\sum_{i\leq n} \dim_{\Bbbk} M_i).
\end{equation}
Let $c$ be a homogenous central element of $B$ of positive degree. 
If $M$ is a finitely generated left graded $B$-module, it follows 
from \eqref{E1.0.2} that
\begin{equation}
\label{E1.0.3}\tag{E1.0.3}
\GKdim M\geq \GKdim M/cM\geq \GKdim M -1.
\end{equation}

Definitions of other standard concepts such as Artin-Schelter 
regularity, Auslander regularity, Cohen-Macaulay property are 
omitted as these can be found in many papers such as 
\cite{Le, CKWZ1, MSm}.

For the first nine sections we consider noncommutative cyclic 
singularities arising from the action of the cyclic group on 
the $(-1)$-skew polynomial ring as follows. 

Let $n$ be a fixed integer $\geq 2$. Let $\n:=\{0,\ldots, n-1\}$.
Note that $\n$ can be identified with the additive group 
${\mathbb Z}_n$. Let $\mathbf{x}$ be the set 
$\{x_0 ,\ldots, x_{n-1}\}$ or $\{x_i\mid i\in {\mathbb Z}_n\}$ 
and $A$ be the $(-1)$-skew polynomial ring $\Bbbk_{-1}[\mathbf{x}]$
as defined in the introduction. Then $A$ is a graded $\Bbbk$-algebra 
with $\mathrm{deg}(x_i)=1$ for each $i$ and we denote $A_j$ the 
$\Bbbk$-subspace of degree $j$ elements of $A$. It is well-known 
that $A$ is noetherian, Artin-Schelter regular, Auslander regular 
and Cohen-Macaulay of global dimension and GKdimension $n$. Let 
$\sigma$ be the cycle $(012\cdots n-1)$ which generates the 
cyclic group $C_n$ of order $n$ as a subgroup of the symmetric 
group $S_n$ (considering $S_n$ as a set of bijections of 
$\n:=\{0,\ldots, n-1\}$). As abstract groups, we have 
${\mathbb Z}_n\cong C_n$. The action of $C_n$ on $A$ is 
determined by its action on generators
$$\sigma \ast x_i= x_{i+1}, \quad \forall \; i \in \n={\mathbb Z}_n.$$
The skew group algebra $A\#C_n$ with respect to this action 
consists of all linear combinations of elements $a \# g$ 
with $a\in A$ and $g\in C_n$, with multiplication given by 
$$(a\# g)( a'\# g') = ag(a')\# gg',$$ 
extended linearly to all of $A\#C_n$. We omit $\#$ if no 
confusion occurs. 

The skew group algebra can be presented in the standard way,
$$
A\#C_n \cong \frac{\Bbbk \langle\mathbf{x}, \sigma\rangle}
{(x_ix_j+x_jx_i , \sigma^n, \sigma x_i - x_{i+1} \sigma)}.
$$
We now describe a different presentation of the above skew group 
algebra, using eigenvectors of the $\sigma$-action, which we 
will use for the rest of the paper.

Since the action of $C_n$ on $A$ is graded, the generating 
subspace $A_1$ is a $C_n$-module. Let $\omega$ be a primitive 
$n$th root of unity and $M_{\omega^j}$ be the simple (hence 
$1$-dimensional) $C_n$-module where $\sigma$ acts by 
multiplication by $\omega^j$. The $\sigma$-action on $A_1$ 
has minimal polynomial $p(X)=X^n-1$, so we can decompose $A_1$ 
as a $C_n$-module as follows
\begin{equation}
\label{E1.0.4}\tag{E1.0.4}
A_1 \cong  \bigoplus_{\gamma=0}^{n-1} M_{\omega^{\gamma}}
\end{equation}
For $\gamma=0,\ldots,n-1$, define the following elements of $A_1
\subseteq A\# C_n$
\begin{equation}
\label{E1.0.5}\tag{E1.0.5}
b_{\gamma} \; := \;  \frac{1}{n} \sum_{i = 0}^{n - 1} 
\omega^{i \gamma} x_i. 
\end{equation}
The following calculation shows that $b_{\gamma}$ is a 
$\omega^{-\gamma}$-eigenvector of $\sigma$,
\begin{equation}
\label{E1.0.6}\tag{E1.0.6}
\sigma \ast b_{\gamma} = \frac{1}{n} \sum_{i = 0}^{n - 1} 
\omega^{i \gamma} x_{i + 1} = \omega^{- \gamma} b_{\gamma}.
\end{equation}
In other words, we have $\Bbbk b_{\gamma}\cong M_{-\gamma}$ as 
$C_n$-modules, so the basis  $\{b_0,\ldots,b_{n-1}\}$ gives 
the $C_n$-module decomposition of $A_1$ in \eqref{E1.0.4}. 
We also define the following idempotent elements
\begin{equation}
\notag
e_{\alpha} \; := \; 
\frac{1}{n} \sum_{i = 0}^{n - 1} (\omega^{\alpha} \sigma)^i
\end{equation}
in $\Bbbk C_n\subseteq A\#C_n$. Let 
$\mathbf{b} := (b_0,\ldots,b_{n-1})$ and $\mathbf{e} := 
(e_0,\ldots,e_{n-1})$. Define the {\it graded commutator}, 
denoted by $[\cdot,\cdot]$, for any homogeneous elements 
$u,v\in A\#C_n$ (or $u,v$ in another graded algebra) by
$$[u,v] = u v - (-1)^{\mathrm{deg}(u)\mathrm{deg}(v)} v u.$$

We have the following Lemma.

\begin{lemma}
\label{xxlem1.1}
Suppose $\deg(b_i)=1$ and $\deg (e_i)=0$ for all
$i\in {\mathbb Z}_n$. The graded algebra $A\#C_n$ can 
be presented as follows
\[A\#C_n \cong
\frac{\Bbbk \langle\mathbf{b},\mathbf{e}\rangle}{
( e_{\alpha}b_{\gamma} - b_{\gamma}e_{\alpha-\gamma}, 
e_{i}e_{j}-\delta_{ij}e_{i}, [b_{0},b_{k}]- [b_{l},b_{k-l}] )}
\]
where $\delta_{ij}$ is the Kronecker delta and indices are 
taken modulo $n$.
\end{lemma}

\begin{proof}
Let $b_i$ be defined as in \eqref{E1.0.5} and let 
\[ r_{kl} = [b_{0},b_{k}]- [b_{l},b_{k-l}]. \]
We first show that the  map 
$$\iota:\Bbbk\langle\mathbf{b}\rangle/(r_{kl}) \to \Bbbk_{-1}[\mathbf{x}]$$ 
is well-defined and is an isomorphism. By \eqref{E1.0.6} 
the elements $b_0,\dots,b_{n-1}$ are eigenvectors for 
the $\sigma$-action on $A_1$ with distinct eigenvalues, 
hence this is a basis for $A_1$, so $\iota$ is an 
isomorphism in degree $1$. To see that $\iota$ is 
well-defined as an algebra map, note that the graded 
commutator of $b_{\gamma}$ and $b_{\delta}$ depends 
only on the sum of $\gamma$ and $\delta$,
\begin{equation}
\label{E1.1.1}\tag{E1.1.1}
[b_{\gamma},b_{\delta}] =  \frac{1}{n^2} \sum_{i, j
= 0}^{n - 1} \omega^{i \gamma + j \delta} [x_i, x_j]
=  \frac{2}{n^2} \sum_{i = 0}^{n - 1} 
\omega^{i (\gamma + \delta)} x_i^2
\end{equation}
So the relations $r_{kl}$ go to zero in $\Bbbk_{-1}[\mathbf{x}]$. 
To show that $\iota$ is an algebra isomorphism, we 
count the number of independent quadratic relations 
in $\mathbf{b}$ and show that this number is equal to 
$\binom{n}{2}$. 

For any fixed $k$, the only linear relations among 
$R_{k} := \{ r_{k 0} ,r_{k1}, \ldots ,r_{k,n-1}\}$ are 
$r_{k 0} =r_{k k} =0$ and $r_{k l} =r_{k,k-l}$.
Define a $C_{2}$-action on $R_{k}$ by $r_{k l} \mapsto r_{k,k-l}$. 
Then the number of independent relations in $R_{k}$ is 
equal to $|R_{k} /C_{2}| -1$.

Case 1: For odd $n$, the $C_{2}$-action has exactly one fixed 
point $r_{k l}$ where $2 l = k \mod n$. Therefore 
$| R_{k} /C_{2} | = ( n+1 ) /2$. The relations in $R_{k}$ are 
independent from the relations in $R_{k'}$ for distinct $k,k'$. 
Since $k$ ranges from $0$ to $n-1$, the total number of
independent relations is equal to
$$n ( | R_{k} /C_{2} | -1 ) = \binom{n}{2} . $$

Case 2: Let $n$ be even.  For odd $k$, the $C_{2}$-action 
has no fixed points. Therefore $|R_{k} /C_{2} | =n/2$. 
If $k$ is even, then the $C_{2}$-action has two fixed
points, coming from the two solutions of $2l= k 
\mod n$. Therefore $|R_{k} /C_{2} | =n/2+1$. By 
considering the odd and even cases separately, we
get that the total number of independent relations 
is equal to
\[\sum_{k \; \text{odd}} ( | R_{k} /C_{2} | -1 ) 
+ \sum_{k \; \text{even}} (| R_{k} /C_{2} | -1 ) 
= \frac{n}{2} \left( \frac{n}{2} -1 \right) +
\frac{n}{2} \left( \frac{n}{2} \right) = \binom{n}{2} . \]
Therefore $\iota$ is an algebra isomorphism.

The isomorphism 
$$\Bbbk C_n\cong \Bbbk \langle \mathbf{e}\rangle 
/(e_i e_j -\delta_{ij}e_i)$$ 
is well-known. The relations between $\mathbf{b}$ and 
$\mathbf{e}$ are obtained as follows
$$
e_{\alpha} b_{\gamma} =  \frac{1}{n} \sum_{i = 0}^{n - 1}
(\omega^{\alpha} \sigma)^i b_{\gamma}
=  \frac{b_{\gamma}}{n} \sum_{i = 0}^{n - 1} 
\omega^{(\alpha - \gamma) i} \sigma^i
=  b_{\gamma} e_{\alpha - \gamma}.
$$
By using the facts
$$\sigma^{i}=\sum_{\alpha} \omega^{-\alpha i} e_{\alpha}$$
and
$$x_j=\sum_{\gamma} \omega^{-\gamma j} b_{\gamma}$$
for $i,j\in {\mathbb Z}_n$, 
it is easy to check that the set of relations
$$\{e_{\alpha} b_{\gamma}=b_{\gamma} e_{\alpha - \gamma}
\mid \alpha, \gamma \in {\mathbb Z}_n\}$$
is equivalent to the set of relations
$$\{ \sigma^i x_j= x_{j+i} \sigma^i \mid i,j \in 
{\mathbb Z}_n\}.$$
This completes the proof.
\end{proof}

We define the elements
\begin{equation}
\label{E1.1.2}\tag{E1.1.2} 
c_j := [b_k,b_{j-k}]
\end{equation}
for all $j\in {\mathbb Z}_n$. 
Equation \eqref{E1.1.1} shows that the definition of 
$c_j$ does not depend on $k$, and while they are 
central elements of $A$, they are not central in 
$A\#C_n$. By the relations in Lemma \ref{xxlem1.1},
we have
\begin{equation}
\notag
e_{\alpha} c_j=c_{j} e_{\alpha-j}
\end{equation}
for all $\alpha, j\in {\mathbb Z}_n$.
As above, we denote by 
$\mathbf{c}=(c_0,c_1,\dots,c_{n-1})$.
Recall that, for a vector $\mathbf{i}
=(i_0,i_1,\ldots,i_{n-1})\in {\mathbb N}^n$,
$$|\mathbf{i}|_1=|i_0|+|i_1|+\cdots +|i_{n-1}|.$$
We will use the following notation
$$\begin{aligned}
\mathbf{b}^{\mathbf{i}}&:= 
     b_0^{i_1}b_1^{i_1}\cdots b_{n-1}^{i_{n-1}},\\
\mathbf{c}^{\mathbf{i}}&:= 
     c_0^{i_1}c_1^{i_1}\cdots c_{n-1}^{i_{n-1}}.
\end{aligned}
$$

\begin{proposition}
\label{xxpro1.2}
For each $r\geq 0$, the set 
$$\mathcal{B}_{r} = \{ \mathbf{b}^{\mathbf{i}}
\mathbf{c}^{\mathbf{j}} \mid \mathbf{i} \in \{ 0,1 \}^{n} ,
\mathbf{j} \in \mathbb{N}^{n} , | \mathbf{i} |_{1} 
+2 |\mathbf{j}|_{1} =r \}$$ 
is a $\Bbbk$-linear basis for $A_{r}$. 
\end{proposition}

\begin{proof} The generating function for
$\mathcal{B}_{r}$, namely,
$g(t)=\sum_{r\geq 0} |\mathcal{B}_{r}| t^r$ is
$$(1+t)^n \frac{1}{(1-t^2)^n}=\frac{1}{(1-t)^n},$$
which agrees with the Hilbert series of $A$. It remains
to show that $\mathcal{B}_{r}$ spans $A_r$ for each $r$. 

Since $\mathbf{b}$ generates $A$, the set 
$\{b_{t_{1}} \cdots b_{t_{r}}\mid {\text{ for different $t_s$}}\}$ 
spans $A_{r}$. Using the relation $c_j=[b_k, b_{j-k}]$
and the fact that $c_j$ are central, we can 
ensure that $b_{t_{1}} \cdots b_{t_{r}}$ is in the linear
span of $\mathcal{B}_{r}$, as required.
\end{proof}

We can extend the above basis for $A_r$ to a basis for 
$(A\#C_n)_r$ by adjoining the $n$ idempotent elements 
coming from $\Bbbk C_n$. Therefore
$$
\mathcal{B}_r\times \mathbf{e} = \{z e_j \mid z\in\mathcal{B}_r,\;
j =0,\dots n-1 \}
$$
and
$$
\mathbf{e}\times \mathcal{B}_r  = \{e_jz \mid z\in\mathcal{B}_r,\;
j =0,\dots n-1 \}
$$
are both $\Bbbk$-linear bases for $(A\#C_n)_r$. The following
is an immediate consequence of Proposition \ref{xxpro1.2}.

\begin{corollary}
\label{xxcor1.3} 
Retain the above notation.
\begin{enumerate}
\item[(1)]
The union $\mathcal{B}=\bigcup_{r\in\mathbb{N}}\mathcal{B}_r$ is a 
$\Bbbk$-linear basis for $A$. 
\item[(2)] 
Both $\mathbf{e}\times \mathcal{B}$ and $\mathcal{B}\times\mathbf{e}$ 
are $\Bbbk$-linear bases for $A\#C_n$.
\item[(3)]
$A\# C_n$ is a finitely generated left and right module 
over the commutative subring $\Bbbk [\mathbf{c}]\subseteq A$.
\end{enumerate}
\end{corollary}

Let $(e_0)\subset A\#C_n$ denote the two sided ideal generated 
by the idempotent $e_0$. We will be concerned with computing 
the GKdimension of the quotient algebra 
$$E:=(A\#C_n)/(e_0).$$ 
Since $e_0$ is the integral of the group algebra $\Bbbk C_n$, 
we obtain that
$$\p(A,C_{n})=\GKdim A -\GKdim E .$$

Let 
\begin{equation}
\label{E1.3.1}\tag{E1.3.1}
\Phi_n :=\{ k \mid c_k^{N_k} \in (e_0) {\text{  for some 
$N_k\geq 0$}}\}.
\end{equation}

The following lemma is easy.

\begin{lemma}
\label{xxlem1.4} Retain the above notation.
\begin{enumerate}
\item[(1)]
Let $\overline{C}$ be the quotient ring
$\Bbbk [\mathbf{c}]/(c_k^{N_k}; k\in \Phi_n)$. Then 
$\GKdim \overline{C} \leq n -|\Phi_n|$.
\item[(2)]
The algebra $E$ is a finitely generated right module 
over $\overline{C}$. As a consequence, 
$$\GKdim E \leq \GKdim \overline{C}\leq n-|\Phi_n|.$$
\item[(3)]
$k\in \Phi_n$ if and only if, for each $\alpha$,
$e_{\alpha} c_k^N\in (e_0)$ for $N\gg 0$.
\end{enumerate}
\end{lemma}

\begin{proof} (1) This is true because 
$\{c_k^{N_k}\mid k\in \Phi_n\}$ is a regular sequence of
$\Bbbk[\mathbf{c}]$.

(2) The first assertion follows from Proposition
\ref{xxpro1.2} (or Corollary \ref{xxcor1.3}(2)).
The consequence follows from \cite[Proposition 8.3.2]{MR}.

(3) If $c_k^{N}\in (e_0)$, then clearly $e_{\alpha}
c_k^{N}\in (e_0)$ for all $\alpha$. The converse
follows from the fact $1=\sum_{\alpha} e_{\alpha}$.
\end{proof}

In the next few sections we provide upper and lower estimates 
for $\GKdim E$.

\section{An upper bound on $\GKdim E $}
\label{xxsec2}

This section is a warm-up for more complicated computations 
to be done in later sections. Fix $n \in \mathbb{N}$, define 
the following functions on $\mathbb{Z}_n$. Let $k$ be in 
${\mathbb Z}_n$. For every $\alpha\in {\mathbb Z}_n$,
\begin{eqnarray*}
  f_k (\alpha) & := & \alpha - k,\\
  g_k (\alpha) & := & 2 \alpha - k.
\end{eqnarray*}
Let $S_k$ be the multiplicative semigroup of 
$\End_{\mathbb{Z}} (\mathbb{Z}_n)$ generated by 
$f_k$ and $g_k$.

\begin{lemma}
\label{xxlem2.1}
For each $s\in S_k$ we have 
$e_{\alpha} c_k^N \in (e_0) + (e_{s (\alpha)})$
for $N\gg 0$.
\end{lemma}

\begin{proof}
We have two simple calculations
\begin{eqnarray*}
e_{\alpha} c_k^N & = & c_k e_{\alpha - k} c_k^{N-1}
=c_k e_{f_{k}(\alpha)} c_k^{N-1}, \qquad \quad {\text{and}}\\
e_{\alpha} c_k^N & = & e_{\alpha} (b_{\alpha} b_{k - \alpha} + 
      b_{k -\alpha} b_{\alpha}) c_k^{N - 1}\\
& = & b_{\alpha} e_0 c_k^{N - 1} b_{k - \alpha} + b_{k - \alpha} 
      e_{2\alpha - k} c_k^{N - 1} b_{\alpha}\\
& = & b_{\alpha} e_0 c_k^{N - 1} b_{k - \alpha} + b_{k - \alpha} 
      e_{g_{k}(\alpha)}c_k^{N - 1} b_{\alpha}, 
\end{eqnarray*}
which imply that $e_{\alpha} c_k^N \in (e_{f_k (\alpha)})$ and 
that $e_{\alpha} c_k^N \in (e_0) + (e_{g_k (\alpha)})$. Since $s$ 
is generated by $f_k$ and $g_k$, the claim follows.
\end{proof}

For fixed $\alpha$, it is easy to see that 
\begin{equation}
\notag
S_k(\alpha)=\{2^s \alpha+tk \mod n \mid s,t\geq 0\}
\subseteq {\mathbb Z}_n.
\end{equation}

\begin{lemma}
\label{xxlem2.2}
Let $k\in {\mathbb Z}_n$ be fixed. If $0\in S_k(\alpha)$ for 
every $\alpha$, then $k\in \Phi_n$.
\end{lemma}

\begin{proof} For each $\alpha$, pick $s\in S_k$ so that 
$s(\alpha)=0$. By Lemma \ref{xxlem2.1}, $e_{\alpha} c_k^N
\in (e_0) + (e_{s(\alpha)})=(e_0)$. The assertion follows by 
Lemma \ref{xxlem1.4}(3).
\end{proof}

Recall from \eqref{E0.6.1} that  
\begin{equation}
\notag
\phi_{2}(n)=\{k \mid 0\leq k\leq n-1, \gcd(k,n)=2^w \;\;
{\text{for some $w\geq 0$}}\}.
\end{equation}

\begin{proposition}
\label{xxpro2.3}
Retain the above notation. 
\begin{enumerate}
\item[(1)] 
If $k=2^w q<n$ such that $q$ is odd and  $(n,q)=1$, then
$k\in \Phi_n$. Equivalently, $\phi_2(n)\subseteq \Phi_n$.
\item[(2)]
$|\Phi_n|\geq |\phi_2(n)|$.
\item[(3)]
$\GKdim (E ) \leq n - |\phi_2(n)|$. As a consequence,
\begin{enumerate}
\item[(3a)]
If $n = 2^j$, then $\GKdim E=0$.
\item[(3b)] 
If $n$ is an odd prime, then $\GKdim E\leq  1$.
\end{enumerate}
\end{enumerate}
\end{proposition}

\begin{proof} Let $n$ be a positive integer such that $n=2^m p$
where $p$ is odd. Then let $|n|_2=m$. 

(1) By Lemma \ref{xxlem2.2}, we need to show that,
for every $\alpha$, there is an $s\in S_{k}$ such that $s(\alpha)=0$. 
Write $\alpha=2^r\beta$ where $r=|\alpha|_2$. Recall that $k=2^w q<n$ 
such that $(p,q)=1$ where $w=|k|_2$. We have two cases, depending on 
the relative magnitudes of $r$ and $w$. 

Case 1: If $r \geq w$, then 
there exists $j$ such that $\alpha = j k$ in $\mathbb{Z}_n$ 
($j=2^{r-w} q^{-1} \beta$ where $q^{-1}$ exists in ${\mathbb Z}_n$), so
$$
f^j_k (\alpha) = \alpha - j k = 0 \qquad {\text{ in ${\mathbb Z}_n$}}.
$$
So we take $s=f^j_k$.

Case 2: If $r < w$, then
$$
g^{w - r}_k (\alpha) = 2^{w - r} \alpha -(2^{w - r } - 1) k,
$$
hence $| g_k^{w - r} (\alpha) |_2 \geq w$. This reduces to the 
first case. 

Hence, in both cases, there is an $s\in S_k$ such that $s(\alpha)=0$ 
as required.

(2) This is an immediate consequence of part (1).

(3) The main assertion follows from part (2) and Lemma \ref{xxlem1.4}(2).
Two consequences are special cases of the main assertion.
\end{proof}

It is easy to see that Theorem \ref{xxthm0.7} is equivalent to 
Proposition \ref{xxpro2.3}(3).

\section{Preparation, part one}
\label{xxsec3}

Recall that $E=(A\# C_n)/(e_0)$. In this section, we reduce 
the problem of computing $\GKdim E$ to that of a right 
quotient module of $A$. Let $\bar{e}_k$ denote the image of 
the idempotent $e_k$ in $E$. This gives a right module 
decomposition 
$$E = \bar{e}_1E \oplus\cdots\oplus \bar{e}_{n-1}E,$$
and it follows that
\begin{equation}
\label{E3.0.1}\tag{E3.0.1}
\GKdim(E )=\max_{1\leq j\leq n-1} \GKdim(\bar{e}_jE ).
\end{equation}
For each $j$, we have the following isomorphism of right 
$A\#C_n$-modules 
\begin{equation}\label{E3.0.2}\tag{E3.0.2}
\bar{e}_j E \cong \frac{e_j (A\#C_n)}{e_j (A\#C_n)\cap(e_0)}.
\end{equation}

Using the basis $\mathbf{e} \times \mathcal{B}$ for $A\#C_{n}$ 
we obtain immediately the right $A$-module isomorphism 
$e_{j} (A\#C_{n})\cong A$ by $e_{j} a \mapsto a$ with inverse 
given by $a \mapsto e_{j} a$. We will use this isomorphism 
to identify $e_{j} (A\#C_{n})$ with $A$ below.

In the following, it will be useful to decompose $A$ according to the 
characters of the $C_n$-action, or equivalently, as modules over the 
invariant subring $A^{C_n}$. Let $R_j$ be the $\Bbbk$-subspace of $A$ 
spanned by the basis consisting of the elements 
$\mathbf{b}^{\mathbf{i}} \mathbf{c}^{\mathbf{j}}$ where 
$(\mathbf{i} + \mathbf{j} ) \cdot \mathbf{v} = j \mod n$ and 
$\mathbf{v}:=( 0,1, \ldots ,n-1 )$. Since $\mathcal{B}$ is an eigenbasis 
with respect to the $\sigma$-action, we have that $R_0$ is the invariant 
subring $A^{C_n}$ and $R_j$ is the $M_{\omega^{-j}}$-isotypic component 
of the $C_n$-action on $A$. This gives an $R_0$-module decomposition
\begin{equation}
\notag
A \cong R_0 \oplus R_1 \oplus \cdots\oplus R_{n-1}.
\end{equation} 
We next find a finite generating set for $R_{j}$.

\begin{lemma}
\label{xxlem3.1}
For each $j=1,\dots,n-1$, define $B_{j}$ to be the set of elements
$\mathbf{b}^{\mathbf{i}} \mathbf{c}^{\mathbf{j}}$ satisfying
\begin{enumerate}
\item[(i)] 
$( \mathbf{i} + \mathbf{j} ) \cdot \mathbf{v} = j
\mod n$, and
\item[(ii)] 
for each nontrivial $\mathbf{b}^{\mathbf{i}'}
\mathbf{c}^{\mathbf{j}'}$ with $\mathbf{i}' \leq \mathbf{i}$
and $\mathbf{j}' \leq \mathbf{j}$ we have
$\mathbf{b}^{\mathbf{i}'} \mathbf{c}^{\mathbf{j}'} \not\in R_{0}$.
\end{enumerate}
Then $B_{j}$ generates $R_{j}$ as a right $R_{0}$-submodule of $A$. 
\end{lemma}

\begin{proof}
By definition, the elements $\mathbf{b}^{\mathbf{i}}
\mathbf{c}^{\mathbf{j}}$ satisfying (i) generate $R_{j}$. 
Now suppose $\mathbf{b}^{\mathbf{i}} \mathbf{c}^{\mathbf{j}}$ 
satisfies (i) but not (ii), that is, there exist some nontrivial
$\mathbf{i}' \leq \mathbf{i}$ and $\mathbf{j}' 
\leq \mathbf{j}$ such that $\mathbf{b}^{\mathbf{i}'} 
\mathbf{c}^{\mathbf{j}'} \in R_{0}$. If
$\mathbf{i}' =0$, then we can write it as 
$\mathbf{b}^{\mathbf{i}}
\mathbf{c}^{\mathbf{j} - \mathbf{j}'}\mathbf{c}^{\mathbf{j}'}$.
If $\mathbf{i}' \neq 0$, then using the commutation relations
\eqref{E1.1.2}, we can move the $\mathbf{b}^{\mathbf{} \mathbf{i}'}$ 
terms, one at a time, to the right side of the expression so that
\[ \mathbf{b}^{\mathbf{i}} \mathbf{c}^{\mathbf{j}} =
\mathbf{b}^{\mathbf{i} - \mathbf{i}'} \mathbf{c}^{\mathbf{j}-\mathbf{j}'} 
\mathbf{b}^{\mathbf{i}'} \mathbf{c}^{\mathbf{j}'}
+ \sum_{\mathbf{k} , \mathbf{l}} \lambda_{\mathbf{k},\mathbf{l}} 
\mathbf{b}^{\mathbf{k}} \mathbf{c}^{\mathbf{l}} \]
where each $\mathbf{k}$ in the summation above satisfies 
$\mathbf{k} < \mathbf{i}$ and $\lambda_{\mathbf{k},\mathbf{l}} 
\in \Bbbk$. In particular, we have expressed 
$\mathbf{b}^{\mathbf{i}} \mathbf{c}^{\mathbf{j}}$ as an $R_{0}$-linear 
combination of terms in $\mathcal{B}$ whose $\mathbf{b}$-exponent 
vector is strictly less than $\mathbf{i}$. By induction on the 
$\mathbf{b}$-exponent vector, we obtain the result.
\end{proof}

\begin{lemma}
\label{xxlem3.2}
Retain the above notation. Suppose 
$1\leq j\leq n-1$ and $0\leq k\leq n-1$.
\begin{enumerate}
\item[(1)]
The intersection $e_{j} (A\#C_{n}) \cap ( e_{0} )$ considered as 
a right ideal in $A$ is generated by $B_{j}$.
\item[(2)]
For $N\geq 0$, $c_k^N\in B_jA$ if and 
only if $e_j c_k^N\in (e_0)$.  
\item[(3)]
$k\in \Phi_n$ if and only if, for each $j$,
$e_{j} c_k^N\in (e_0)$ for some $N\gg 0$; and 
if and only if, for each $j$,
$c_k^N\in B_jA$ for some $N\gg 0$.
\end{enumerate}
\end{lemma}

\begin{proof}
(1) Using the fact $A\# C_n=\sum_i A e_i=\sum_i e_i A$, one sees that
every element $f\in (e_0):= (A\# C_n) e_0 (A\# C_n)$ can be written 
as a linear combination of terms $ue_0v$ where $u,v\in\mathcal{B}$. 
Without loss of generality let $f= ue_0 v$ where $u,v \in \mathcal{B}$. 
If, in addition, $f\in e_j (A\#C_n)$,  then 
$$f= e_j f= e_j u e_0 v= u e_{j-\gamma} e_0 v=\begin{cases} u e_0 v & j=\gamma\\
0 & j\neq \gamma\end{cases}$$
where $u\in R_{\gamma}$. Hence we can assume that $j=\gamma$ and 
$u\in R_j=B_j R_0$ by Lemma \ref{xxlem3.1}. Since elements of $R_0$ 
commute with $e_0$, we can actually assume that $u\in B_j$. Finally,
$ue_0 v=e_j uv$ since $u\in B_j$.

(2) This follows from part (1).

(3) This follows from part (2) and Lemma \ref{xxlem1.4}(3).
\end{proof}

\medskip
Identify $e_j (A\# C_n)$ with $A$ and combining Lemma \ref{xxlem3.2} 
and \eqref{E3.0.2}, we get
\begin{equation}
\label{E3.2.1}\tag{E3.2.1}
\bar{e}_jE \cong A/B_jA =A/R_jA.
\end{equation}
We can say more: Lemma \ref{xxlem3.4} below finds a sufficient 
condition for when these quotients are isomorphic.

\begin{definition}
\label{xxdef3.3} 
Let $\lambda\in {\mathbb Z}$ be an integer with $\gcd(\lambda,n)=1$. 
Let $f_{\lambda} :A \longrightarrow A$ be the algebra 
map determined by 
$$f_{\lambda}(b_{i})= b_{\lambda i}$$ 
for all $i\in {\mathbb Z}_n$. To see this is an algebra 
homomorphism, note that 
$$f_{\lambda} ([ b_{0} ,b_{j} ] - [ b_{r} ,b_{j-r} ] ) 
= [b_{0} ,b_{\lambda j}]-[b_{\lambda r} ,b_{\lambda j-\lambda r}].$$ 
Since $\lambda$ is invertible in ${\mathbb Z}_n$, 
$f_{\lambda}$ is an algebra automorphism of $A$. It is 
easy to check that $f_{\lambda}(x_i)=x_{a i}$ where $a=\lambda^{-1}$ 
in ${\mathbb Z}_n$. 
\end{definition}

\begin{lemma}
\label{xxlem3.4}
For any positive integer $\lambda$ with $\gcd(\lambda,n)=1$, we 
have the following isomorphism of $\Bbbk$-vector spaces 
$$\bar{e}_{j} E  \cong \bar{e}_{\lambda j} E .$$ 
In particular, if $n$ is prime, then for each $j=2, \ldots ,n-1$, 
we have $\bar{e}_{1} E  \cong \bar{e}_{j} E$.
\end{lemma}

\begin{proof}
Let $f_{\lambda} :A \longrightarrow A$ be the algebra isomorphism 
defined in Definition \ref{xxdef3.3}. Now 
$$f_{\lambda} (R_{j} ) =R_{\lambda j},$$ hence 
$$\bar{e}_{j}
E \cong A/R_{j} A\cong A/R_{\lambda j} A\cong \bar{e}_{\lambda j}
E $$ 
as $\Bbbk$-vector spaces. 
\end{proof}

Since GKdimension of a finitely generated $A$-module is only
dependent on its Hilbert series \eqref{E1.0.2}, we have the 
following immediate consequences.

\begin{corollary}
\label{xxcor3.5} Retain the above notation.
\begin{enumerate}
\item[(1)]
For any $0<j<n$, we have the following lower bound 
for $\GKdim(E)$ 
$$\GKdim(E)\geq \GKdim (A/B_jA).$$
\item[(2)]
We have
$$\GKdim(E)=\max_j \GKdim (A/B_jA)$$
where $j$ ranges over positive integers less than $n$ 
that divide $n$.  
\item[(3)]
If $n$ is prime, then 
$$\GKdim(E) = \GKdim (A/B_1A).$$
\end{enumerate}
\end{corollary}

For the rest of this section we will consider two distinct 
values of $n$, with one a factor of the other, and arguments 
will involve two particular natural algebra homomorphisms 
between the $(-1)$-skew polynomial rings of these different 
dimensions.

We fix two integers $m$ and $n$ such that $m$ divides $n$. 
Let $A$ (resp. $\tilde{A}$) denote the $(-1)$-skew polynomial 
ring of dimension $n$ (respectively, $m$). Usually we use 
$\tilde{\quad}$ to denote the corresponding notation for 
the algebra $\tilde{A}$. For example, since we use 
$\mathbf{b}$ for the generating set for $A$
(see \eqref{E1.0.5}), then we use $\tilde{\mathbf{b}}$ 
to denote the corresponding generating set for $\tilde{A}$.  
Recall from the proof of Lemma \ref{xxlem1.1}, the algebra $A$
is determined completely by the set of relations of the form
\begin{equation}
\label{E3.5.1}\tag{E3.5.1}
r_{kl}: [b_0,b_k]-[b_l,b_{k-l}]=0
\end{equation}
for all $k,l\in {\mathbb Z}_n$. Similarly for the algebra
$\tilde{A}$.

\begin{definition}
\label{xxdef3.6}
Suppose $m$ divides $n$. There is a surjective  homomorphism 
$$\pi_{n,m}:A \to \tilde{A}$$ 
determined by sending $b_j \mapsto \tilde{b}_{j}$, where in 
the $\tilde{\mathbf{b}}$ variables the indices are taken 
modulo $m$. Since $m$ divides $n$, $\pi_{n,m}$ maps any 
relation of $A$ of the form \eqref{E3.5.1} to a relation of 
$\tilde{A}$. Therefore $\pi_{n,m}$ is an algebra homomorphism.
The surjectivity of $\pi_{n,m}$ follows from the fact that 
it is surjective in degree 1. 
\end{definition}

\begin{lemma}
\label{xxlem3.7}
Suppose that $m$ divides $n$. Then 
$\GKdim (E)\geq \GKdim (\tilde{E})$. 
\end{lemma}

\begin{proof}
Let $\tilde{R}_j\subset \tilde{A}$, for $j=0, \ldots, m-1$, 
be defined as in the beginning of Section \ref{xxsec3} for 
the algebra $\tilde{A}$ with $m$ variables. Since 
$\pi_{n,m}(R_j) \subset \tilde{R}_j$, for $j=0,\ldots, m-1$, 
we get a surjective homomorphism 
$A/R_j A\to\tilde{A}/\tilde{R_j}\tilde{A}$. Then 
$$\GKdim (\tilde{E}) 
=\max_{j=1}^{m-1} \{\GKdim \tilde{A}/\tilde{R}_j\tilde{A}\}
\leq \max_{j=1}^{m-1} \{\GKdim A/R_jA\}
\leq \GKdim (E ).$$
\end{proof}

Lemma \ref{xxlem3.7} will be used in the proof of Theorem 
\ref{xxthm0.4}. 

For the proof the main result (Theorem \ref{xxthm0.2}), we 
need to consider another homomorphism. As before, let $m$ 
and $n$ be two integers such that $m$ divides $n$. Write 
$q=n/m$. 

\begin{definition}
\label{xxdef3.8}
Suppose $m$ divides $n$ and write $q=n/m$. Let 
$$\theta_{m,n}: \tilde{A}\to A$$
be an algebra homomorphism determined by 
$\theta_{m,n}(\tilde{b}_i)= b_{qi}$ for all 
$i\in {\mathbb Z}_m$. Since $\theta_{m,n}$ maps the relation 
$\tilde{r}_{kl}$ of $\tilde{A}$ of the form \eqref{E3.5.1} 
to $r_{qk,ql}$ of $A$, $\theta_{m,n}$ is an algebra homomorphism. 
\end{definition}

We have the following easy lemma.

\begin{lemma}
\label{xxlem3.9} Suppose that $m$ divides $n$ and write $q=n/m$.
Let $N$ be a positive integer. Then
\begin{enumerate}
\item[(1)]
$\theta_{m,n}(\tilde{R}_j)\subseteq R_{qj}$ for all 
$j\in {\mathbb Z}_m$.
\item[(2)]
$\theta_{m,n}(\tilde{c}_j)=c_{qj}$ for all $j\in 
{\mathbb Z}_m$.
\item[(3)]
If $\tilde{c}^N_i\in \tilde{R}_j\tilde{A}$ for some 
$i,j\in {\mathbb Z}_m$ and $N\geq 0$, then $c^N_{qi}\in R_{qj}A$.
\end{enumerate}
\end{lemma}

\section{Proof of Theorem \ref{xxthm0.4}}
\label{xxsec4}

We first show that for $n=3,5$ the GKdimension of $A/B_1A$ 
is equal to $1$. Hence the GKdimension of $E$ is also 
equal to $1$ by Corollary \ref{xxcor3.5}(3). It turns out 
that we can use these two cases to infer that the GKdimension 
of $E $ is positive whenever $3$ or $5$ divides $n$.

\begin{proposition}
\label{xxpro4.1}
Let $n=3$.
\begin{enumerate}
\item[(1)]
$B_1A\cong (b_1A+c_1A)$ as right $A$-modules.
\item[(2)]
$\GKdim(A/B_1A)=1$.
\item[(3)]
$\GKdim(E )=1$.
\end{enumerate}
\end{proposition}

\begin{proof}
(1) Recall that the definition of $B_j$ is given in Lemma \ref{xxlem3.1}. 
By definition one can easily check that 
$B_{1} = \{ b_{1}, c_{1}, c_2^2, b_2 c_2\}$. Note that
$c_{2} =2b_{1}^{2}$ so $c_2^2=4 b_1^4$ and $b_{2} c_{2} =2b_{1}^{2}b_{2}$, 
hence $B_{1} A=b_{1} A+c_{1} A$.

(2) Since $c_1$ is central in $A$ the quotient $W=A/(c_1)$ has 
the structure of a $\Bbbk$-algebra with $\GKdim(W) = \GKdim(A)-1=2$. 
Then $A/B_1A\cong W/b_1W$, so 
$$\GKdim A/B_1A =\GKdim(W/b_1W)\geq \GKdim(W)-1=1.$$
On the other hand, $b_1W \supseteq c_2W$. Then 
$$ \GKdim A/B_1A =\GKdim(W/b_1W)\leq \GKdim(W/c_2W)=\GKdim (W)-1=1.$$
The assertion follows.

(3) The assertion follows from Corollary \ref{xxcor3.5}(3) and 
part (2).
\end{proof}

\begin{proposition}
\label{xxpro4.2}
Let $n=5$.
\begin{enumerate}
\item[(1)]  
$B_{1} A \subseteq I$ where $I= ( b_{1} A+b_{2} A+c_{1}
A+c_{2} A+c_{3} A+c_{4} A)$. 
\item[(2)]
$\GK (A/I) =1$.
\item[(3)] 
$\GKdim (E )=\GKdim (A/B_{1}A) =1$. 
\end{enumerate}
\end{proposition}

\begin{proof}
(1) An element $\mathbf{b}^{\mathbf{i}} \mathbf{c}^{\mathbf{j}} \in
B_{1}$ with $\mathbf{j} \neq 0$ is clearly in $I$. To verify the 
inclusion, it suffices to show that if $\mathbf{b}^{\mathbf{i}} 
\in B_{1}$ with $\mathbf{i} \in \{ 0,1 \}^n$ then 
$\mathbf{b}^{\mathbf{i}} \in I$. There are two such elements 
$b_{1}$ and $b_{2} b_{4}$ and these are both in $I$.

(2) Let $J$ be the two sided ideal of $A$ generated by central 
elements $c_1,c_2,c_3,c_4$. Let $\beta_{j}$ (resp.
$\gamma_{0}$) denote the image of $b_{j}$ (resp. $c_{0}$) 
in $A/J$. Then $A/J$ is a finitely generated left $\Bbbk[b_0]$-module. 
Since $A/J$ has no $\beta_0$-torsion, it is actually a free module 
over $\Bbbk [\beta_0]$. Moreover, $\beta_0$ skew-commutes with 
the other $\beta_i$'s, so a $\Bbbk[\beta_0]$-basis for $A/J$ is 
given by squarefree monomials (with respect to the lexicographical 
ordering) in $\beta_1,\dots,\beta_4$. Using this basis, we see that 
$$A/I\cong \frac{A/J}{\beta_1(A/J)+\beta_2(A/J)} \cong
\Bbbk [\gamma_0] \oplus 
\Bbbk [\gamma_0] \beta_3 \oplus
\Bbbk [\gamma_0] \beta_4 \oplus
\Bbbk [\gamma_0] \beta_3\beta_4.$$
Hence $\GK ( A/I ) =1$. 

(3) By part (1), the map $A/B_{1} A \rightarrow A/I$ is
surjective, so $\GKdim ( A/B_{1} A ) \geq 1$. 
By Proposition \ref{xxpro2.3}(3b), $\GKdim (E )\leq 1$.
Combining with Corollary \ref{xxcor3.5}(3), 
we have $\GKdim (E )=\GKdim(A/B_{1}A)=1$.
\end{proof}

Now we are ready to prove Theorem \ref{xxthm0.4}.

\begin{proof}[Proof of Theorem \ref{xxthm0.4}]
Retain notation as in Lemma \ref{xxlem3.7}.
Take $m=3$ or $5$ (two different cases). By Proposition 
\ref{xxpro4.1}(3) and \ref{xxpro4.2}(3), 
$\GKdim (\tilde{E})=1$. By Lemma \ref{xxlem3.7},
$\GKdim (E)\geq 1$. Hence $\p(A, G)\leq n-1$. 
By \cite[Theorem 3.10]{MU1}, $A^G$ is not a graded 
isolated singularity.
\end{proof}

\section{Preparation, part two}
\label{xxsec5}
In Section \ref{xxsec7} we will prove Theorem \ref{xxthm0.2}. 
We need to do several reduction steps, some of which 
are given in this section. First we fix some convention
throughout the rest of the paper.

\begin{convention}
\label{xxcon5.1} 
Let $n$ denote a fixed integer $\geq 2$. Letters such as 
$i,j,k$ denote elements in ${\mathbb Z}_n$. Usually these 
take values in $[0,1,2,\ldots, n-1]$. However $0$ is identified 
with $n$. If we use induction, the induction process starts with $1$ 
and ends with $n$ (then $n$ is identified with $0$). So, when 
we use induction on the integer $i$ it will take values in $[1,2,\ldots, n]$. 

In Section \ref{xxsec5}, we only use $i$ and $j$.
\end{convention}

Some ideas in this section 
have appeared in previous sections, but we will do finer analysis. 
In order to prove Theorem \ref{xxthm0.2}, we seek to show that 
for every $j$ with $1\leq j<n$, the $\Bbbk$-vector space $A/B_j A$ is 
finite dimensional. It is necessary and sufficient to show that 
for every $i$, the element $c_i^N\in B_j A$ for some $N \geq 0$. 

\begin{definition}
\label{xxdef5.2} Retain notation above. 
\begin{enumerate}
\item[(1)]
We say $c_i$ is {\it nilpotent} in $A/B_jA$ if
$c_i^N\in B_j A$ for some $N \geq 0$. In this case 
we write $i\in \Psi^{[n]}_j$.
\item[(2)]
We say $n$ is {\it admissible} if, for every all
$i$ and $j$, $i\in \Psi^{[n]}_j$, or
equivalently, $\GKdim E=0$, see Lemma \ref{xxlem5.3}(1) 
below. 
\end{enumerate}
\end{definition}

Note that it is automatic that $i\in \Psi^{[n]}_0$. 
Therefore usually we only consider the case when 
$1\leq j\leq n-1$. We start with some initial analysis 
and easy reductions.

\begin{lemma}
\label{xxlem5.3} Retain notation above. 
\begin{enumerate}
\item[(1)]
$n$ is admissible if and only if $\GKdim (E)=0$. In this 
case, $\p(A,G)=n$.
\item[(2)]
If $i\in \Psi^{[n]}_j$ for all $i$ and all divisors $j\mid n$
with $1\leq j<n$, then $n$ is admissible.
\item[(3)]
If $m$ is a factor of $n$ and $n$ is admissible, 
then $m$ is admissible.
\end{enumerate}
\end{lemma}

\begin{proof} (1) The assertion follows from 
\eqref{E3.0.1}, \eqref{E3.2.1} and the fact 
that 
$${\text{$\GKdim (A/B_jA)=0$ if and only if 
$i\in \Psi^{[n]}_j$ for all $i$.}}$$

(2) This is Corollary \ref{xxcor3.5}(2). 

(3) This follows from part (1) and Lemma \ref{xxlem3.7}.
\end{proof}

\begin{lemma}
\label{xxlem5.4}
Retain notation above. 
\begin{enumerate}
\item[(1)]
If $n=mq$ and $i\in \Psi^{[m]}_{j}$, then 
$iq\in \Psi^{[n]}_{jq}$.
\item[(2)]
Let $j$ be a divisor of $n$. 
If $\gcd(i,n) = \gcd(i,j)$, or $\gcd(i,n) | j$,
then $i\in \Psi^{[n]}_j$.
\end{enumerate}
\end{lemma}

\begin{proof}
(1) This is Lemma \ref{xxlem3.9}(3).

(2) Let $q=\gcd(i,n)$. Then $q=\gcd(i,j)=\gcd(i,j,n)$.
By part (1), we might assume that $q=1$. In this case,
$i$ is invertible in ${\mathbb Z}_n$. Let $s$ be the 
inverse of $i$ in ${\mathbb Z}_n$. Then there is a
$t:=j s$ such that $j=ti$ in ${\mathbb Z}_n$. 
In this case $c_i^t \in R_j= B_j A$ as desired.
\end{proof}

Parts (1) to (3) of the next lemma are in fact a slightly 
different version of Lemma \ref{xxlem2.1}.

\begin{lemma}
\label{xxlem5.5}
Retain notation above. 
\begin{enumerate}
\item[(1)]
If $i\in \Psi^{[n]}_{2j-i}$, then $i\in \Psi^{[n]}_j$.
\item[(2)]
If $i\in \Psi^{[n]}_{j-i}$, then $i\in \Psi^{[n]}_{j}$.
\item[(3)]
If $i\in \Psi^{[n]}_{2^s j+t i}$ for some integers
$s,t\geq 0$, then $i\in \Psi^{[n]}_{j}$.
\item[(4)]
Suppose that every proper divisor of $n$ is admissible. 
If $\gcd(i,n)$ is even, then $i\in \Psi^{[n]}_j$.
\end{enumerate}
\end{lemma}

\begin{proof}
(1) If $c_i^N \in B_{2j-i} A$, then we can show $c_i^{N+1} 
\in B_j A$ as follows:
$$\begin{aligned} 
c_i^{N+1} &= c_i c_i^N =(b_j b_{i-j} + b_{i-j} b_j) c_i^N \\
&= b_j (b_{i-j} c_i^N) + (b_{i-j} c_i^N) b_j \\
&\in B_j (b_{i-j} c_i^N) + (b_{i-j} B_{2j-i}A) b_j \\
& \subseteq R_j A=B_j A. 
\end{aligned}
$$

(2) By definition, we have $c_i^N\in B_{j-i}A$ for some $N>0$. Then 
$$c_i^{N+1}=c_i c_i^N\in c_i B_{j-i}A\subseteq B_{j}A.$$
The assertion follows. 

(3) Applying the statement of (2) multiple 
times, we have that $i\in \Psi^{[n]}_{2^s j-i}$. 
By part (1), we have $i\in \Psi^{[n]}_{2^{s-1}j}$. 
The assertion follows by induction on $s$.

(4) Let $i=2 i'$ and $n=2 n'$. 
Since $n'$ is admissible, $i'\in \Psi^{[n']}_{j-i'}$. By Lemma 
\ref{xxlem5.4}(1), $i\in \Psi^{[n]}_{2j-i}$. The assertion follows 
from part (1).
\end{proof}

\section{Preparation, part three}
\label{xxsec6}
Recall from \eqref{E1.3.1} that
\begin{equation}
\notag
\Phi_n :=\{ k \mid c_k^{N_k}\in (e_0) {\text{  for some 
$N_k\geq 0$}}\}.
\end{equation}
For each $k\in \Phi_n$, there exists $N_k \ge 0$ such that 
$c_k^{N_k}=0$ in $E=(A\# C_n)/(e_0)$. It is easy to see 
that the set $\Phi_n$ satisfies the condition in the 
following definition.

\begin{definition}
\label{xxdef6.1}
A subset of $\Phi\subseteq {\mathbb Z}_n$ is called {\it special} 
if $k\in \Phi$ if and only if $\lambda k\in \Phi$ for all 
invertible elements $\lambda\in {\mathbb Z}_n$. In this case, the 
ideal $c_{\Phi}:=\langle c_k \mid k\in \Phi\rangle$ of $A$ is 
called the {\it special ideal} of $A$ associated to $\Phi$. 
\end{definition}

Here are some examples of special subsets:
\begin{enumerate}
\item[(1)]
$\Phi=\emptyset$ (in which case, $c_\Phi = 0$).
\item[(2)]
$\Phi=\Phi_n$ as in \eqref{E1.3.1}.
\item[(3)]
$\Phi=\phi_2(n)$ as in \eqref{E0.6.1}.
\item[(4)]
$\Phi=\{1,2,\ldots, n-1\}$.
\item[(5)]
$\Phi=\{0,1,2,\ldots, n-1\}$ (in which case, 
$c_{\Phi}=\{c_k\mid 0\leq k \leq n-1\}$).
\end{enumerate}

Fix one special ideal $c_{\Phi}$ of $A$, and write $\overline{A}=A/c_{\Phi}$.
Clearly, $C_n$ acts on $\overline{A}$. Let $\overline{E}$ be the algebra 
$(\overline{A}\#C_n)/(e_0)$. The following lemma shows that it is useful 
to pass into the quotient rings.

\begin{lemma}
\label{xxlem6.2} Retain the notation above and suppose that $\Phi=\Phi_n$.
Then 
$$\GKdim E=\GKdim \overline{E}.$$
\end{lemma}

\begin{proof}
Since $E$ is noetherian, 
$$\GKdim E=\max_{{\mathfrak p}}\GKdim E/{\mathfrak p}$$
where the $\max$ runs over all prime ideals ${\mathfrak p}$ of $E$.  
Since $c_k$, for each $k\in \Phi_n$, is normal and nilpotent 
in $E$, we have $c_k\in {\mathfrak p}$ for each prime ${\mathfrak p}$. 
Hence $E/{\mathfrak p}$ is annihilated by the ideal $c_{\Phi}$.
As a consequence,
$$\GKdim E/{\mathfrak p}=\GKdim E/{\mathfrak p}
\otimes A/c_{\Phi}\leq \GKdim E\otimes A/c_{\Phi}=\GKdim \overline{E}.$$
This implies that $\GKdim E\leq \GKdim \overline{E}$. It is
clear that $\GKdim E\geq \GKdim \overline{E}$. The assertion 
follows.
\end{proof}

Next we repeat some arguments in Section \ref{xxsec3}.
Going back to a general fixed special ideal (not necessarily
associated to $\Phi_n$),
by abuse of notation, let $\bar{e}_k$ also denote the image 
of the idempotent $e_k$ in $\overline{E}$. Then we have a right 
$\overline{E}$-module decomposition 
$$\overline{E} = \bar{e}_1 \overline{E} \oplus\cdots\oplus 
\bar{e}_{n-1}\overline{E}$$
and it follows that
\begin{equation}
\notag
\GKdim(\overline{E})=\max_{1\leq j\leq n-1} \GKdim(\bar{e}_j\overline{E}).
\end{equation}
For any $j$, we have the following isomorphism of right 
$\overline{A}\# C_n$-modules 
\begin{equation}
\notag
\bar{e}_j \overline{E} \cong \frac{e_j (\overline{A}\# C_n)}
{e_j (\overline{A}\# C_n)\cap(e_0)}.
\end{equation}

We recycle the letters $x_i, b_i, c_i$ for $\overline{A}$
(with some of $c_i=0$ in $\overline{A}$).
There is a right $\overline{A}$-module isomorphism 
$e_{j} (\overline{A}\#C_{n})\cong \overline{A}$ by $e_{j} a \mapsto a$ 
with inverse given by $a \mapsto e_{j} a$. So we will 
identify $e_{j} (\overline{A}\#C_{n})$ with $\overline{A}$ below.

Let $\overline{\mathcal{B}}_j$ (respectively, 
$\overline{\mathcal{B}}$, $\overline{B}_j$) 
be defined as in Proposition \ref{xxpro1.2} (respectively, 
Corollary \ref{xxcor1.3}, Lemma \ref{xxlem3.1})
after removing all $\{c_k \mid k\in \Phi\}$. Let $\overline{R}_j$ be 
the $M_{\omega^{-j}}$-isotypic component of the $C_n$-action on $\overline{A}$.
We have an $R_0$-module decomposition
\begin{equation}
\notag
\overline{A} \cong \overline{R}_0 \oplus \overline{R}_1 
\oplus \cdots\oplus \overline{R}_{n-1}
\end{equation} 
where $\overline{R}_0=(\overline{A})^{C_n}$. The following is 
an $\overline{A}$-version of Lemma \ref{xxlem3.2}.

\begin{lemma}
\label{xxlem6.3}
Retain the notation above. We are working in the 
algebra $\overline{A}\# C_n$.
\begin{enumerate}
\item[(1)]
The intersection $e_{j} (\overline{A}\#C_{n}) \cap ( e_{0} )$ considered as 
a right ideal in $\overline{A}$ is generated by $\overline{B}_{j}$.
\item[(2)]
For $N\geq 0$, we have $c_k^N\in \overline{B}_j\overline{A}$ if and 
only if $e_j c_k^N\in (e_0)$.  
\item[(3)]
If there is an integer $N\geq 0$ such that, for each $j$, we have
$e_{j} c_k^N\in (e_0)$, then $c_k^N=0$ in $\overline{E}$.
If, in addition, we have $\Phi=\Phi_n$, then $c_k^N=0$ in $E$, or equivalently,
$c_k=0$ in $\overline{E}$.
\end{enumerate}
\end{lemma}

\begin{proof}
For (1) and (2), see the proof of Lemma \ref{xxlem3.2}.

(3) Since $1=\sum e_j$, we have $c_k^N\in (e_0)$. This means that $c_k^N=0$ in
$\overline{E}$. 

Now assume $\Phi=\Phi_n$.
Since $c_k$ is normal, $c_k\in {\mathfrak q}$
for every prime ideal ${\mathfrak q}$ of $\overline{E}$. By the proof of 
Lemma \ref{xxlem6.2}, every prime quotient $E/{\mathfrak p}$ of $E$
is isomorphic to $\overline{E}/{\mathfrak q}$ for some prime ideal 
${\mathfrak q}$ of $\overline{E}$. This implies that $c_k$ is zero 
in $E/{\mathfrak p}$, consequently, $c_k$ is nilpotent in $E$, 
or $c_i^{N'}=0$ in $E$ for some $N'$. The assertion follows.
\end{proof}

We also have the $\overline{A}$-versions of Lemma \ref{xxlem3.4} and 
Corollary \ref{xxcor3.5}. The statements are the following and proofs 
are omitted.

\begin{lemma}
\label{xxlem6.4}
For any positive integer $\lambda$ with $\gcd(\lambda,n)=1$, we have the
following isomorphism of $\Bbbk$-vector spaces 
$$\bar{e}_{j} \overline{E}  \cong \bar{e}_{\lambda j} \overline{E}.$$ 
In particular, if $n$ is prime, then for each
$j=2, \ldots ,n-1$, we have $\bar{e}_{1} \overline{E}\cong 
\bar{e}_{j} \overline{E}$.
\end{lemma}

\begin{lemma}
\label{xxlem6.5} Retain the above notation.
\begin{enumerate}
\item[(1)]
For any $0<j<n$, we have the following lower bound 
for $\GKdim \overline{E}$ 
$$\GKdim \overline{E}\geq \GKdim (\overline{A}/\overline{B}_j\overline{A}).$$
\item[(2)]
We have
$$\GKdim \overline{E}=\max_j \GKdim (\overline{A}/\overline{B}_j\overline{A})$$
where $j$ ranges over positive integers less than $n$ 
that divide $n$.  
\item[(3)]
If $n$ is prime, then 
$$\GKdim \overline{E} = \GKdim (\overline{A}/\overline{B}_1\overline{A}).$$
\end{enumerate}
\end{lemma}

One advantage of 
working with $\overline{A}$ is that 
\begin{equation}
\label{E6.6.1}\tag{E6.6.1}
b_i b_{k-i}= b_i b_{k-i}+b_{k-i}b_i-b_{k-i}b_i=c_k-b_{k-i}b_i=-
b_{k-i}b_i
\end{equation}
for all $k\in \Phi$. 

Similar to Definition \ref{xxdef5.2}(1),
we say $c_i\in \overline{A}$ is {\it nilpotent} in 
$\overline{A}/\overline{B}_j\overline{A}$ if
$c_i^N\in \overline{B}_j \overline{A}$ for some $N \geq 0$. 
In this case we write $i\in \overline{\Psi}^{[n]}_j$.
Now we are ready to take care 
of Theorem \ref{xxthm0.2} when $n$ is prime.

\begin{proposition}
\label{xxpro6.6}
Suppose $n\geq 2$ is neither 3 nor 5.
\begin{enumerate}
\item[(1)]
Suppose $\Phi\supseteq \{1,\ldots, n-1\}$. Then 
$0\in \overline{\Psi}^{[n]}_1$. 
\item[(2)]
If $n$ is prime, then $\GKdim E=0$. Consequently, 
$n$ is admissible. 
\end{enumerate}
\end{proposition}

\begin{proof} (1)
We start with the trivial observation that if $u,v \in 
\overline{B}_{1}\overline{A}$, then $[u,v] \in 
\overline{B}_{1}\overline{A}$. The strategy is to start 
with $b_1 \in \overline{B}_{1}\overline{A}$ 
and apply a sequence of graded commutations with 
selected elements of $\overline{B}_{1}\overline{A}$ to 
obtain $c_0^{n-1} \in \overline{B}_{1}\overline{A}$. 
Note that in $\overline{A}$, we have $[b_{i},b_{j}]=0$ 
unless $i+j=0 \mod n$.

\bigskip
\noindent
{\bf Claim 1:} Suppose that $0\leq j\leq n-1$ and 
$2j+1 \neq 0 \mod n$ (or $2j+1\neq n$). If 
$c_0^s b_j\in \overline{B}_{1}\overline{A}$ for some $s$, then
$c_0^{s+1} b_{j+1}\in \overline{B}_{1}\overline{A}$

\noindent
{\bf Proof of Claim 1:}
First of all, $b_{j+1} b_{n-j} \in \overline{B}_{1}\overline{A}$ 
(which is not central). Since $2j+1 \neq 0 \mod n$, 
$b_{j} b_{j+1} =- b_{j+1} b_{j}$ in $\overline{A}$
\eqref{E6.6.1}. We obtain 
\begin{eqnarray*}
- [c_0^s b_{j}, b_{j+1} b_{n-j} ] & = & -  c_0^s b_{j} b_{j+1}
b_{n-j} + c_0^s b_{j+1} b_{n-j} b_{j}\\
& = & c_0^s(b_{j+1} b_{j} b_{n-j} + b_{j+1} b_{n-j}b_{j})\\
& = & c_0^{s}b_{j+1} c_{0}=c_0^{s+1}b_{j+1}.
\end{eqnarray*}
The assertion follows.

\bigskip
\noindent
{\bf Claim 2:} Suppose otherwise that $2j+1=0\mod n$, so that
$2j+1=n$. If $c_0^s b_j\in \overline{B}_{1}\overline{A}$, then
$c_0^{s+2} b_{j+2}\in \overline{B}_{1}\overline{A}$.

\noindent
{\bf Proof of Claim 2:} Under the hypothesis of $j$, we have  
that $n$ is odd and that $b_{j+1} b_{j+2} b_{n-1} \in 
\overline{B}_{1}\overline{A}$. Given that $n \neq 3,5$, we 
have $n\geq 7$, and consequently, 
$$j+2= (n+3)/2 <n-1,$$ 
so that the indices in $b_{j+1} b_{j+2} b_{n-1}$ are strictly 
increasing. This gives the following commutator 
computation in $\overline{B}_{1}\overline{A}$,
\begin{eqnarray*}
[ c_0^s b_{j} , b_{j+1} b_{j+2} b_{n-1} ] & = & c_0^s(b_{j}
b_{j+1} b_{j+2} b_{n-1} + b_{j+1} b_{j+2}
b_{n-1} b_{j})\\
& = & c_0^s(b_{j}
b_{j+1} b_{j+2} b_{n-1} + b_{j+1} b_{j} b_{j+2}
b_{n-1} )\\
& = & c_0^{s+1} b_{j+2} b_{n-1} .
\end{eqnarray*}
We apply the one additional commutator to get
\[ [ [ c_0^s b_{j} , b_{j+1} b_{j+2} b_{n-1} ] , b_{1} ] =
[ c_0^{s+1} b_{j+2} b_{n-1}, b_{1}] = c_0^{s+2} b_{j+2}. \]
Therefore, if $c_0^s b_j \in \overline{B}_{1}\overline{A}$ 
and $2j+1=n$, then $c_0^{s+2}b_{j+2} \in \overline{B}_{1}\overline{A}$.

\bigskip
\noindent
{\bf Claim 3:} $c_0^{n-1}\in \overline{B}_{1}\overline{A}$.

\noindent
{\bf Proof of Claim 3:} 
Starting with $b_{1}$, we can apply {\bf Claim 1} $(n-2)$-times 
to get $c_0^{n-2} b_{n-1} \in \overline{B}_{1}\overline{A}$ whenever
$n$ is even. Hence $c_0^{n-1}=[c_0^{n-2} b_{n-1}, b_1]
\in \overline{B}_{1}\overline{A}$ as required.

If $n$ is odd, we apply {\bf Claim 1} $(j_0-1)$-times 
to get $c_0^{j_0-1}b_{j_0} \in \overline{B}_{1}\overline{A}$ where 
$j_0=\frac{n-1}{2}$. Next we apply {\bf Claim 2} to get $c_0^{j_0+1}b_{j_0+2} 
\in \overline{B}_{1}\overline{A}$. Then apply {\bf Claim 1}
again $(j_0-2)$-times to get $c_0^{n-2}b_{n-1} \in 
\overline{B}_{1}\overline{A}$. Finally we have 
\[ c_{0}^{n-1}=[c_0^{n-2}b_{n-1} , b_{1} ] 
\in \overline{B}_{1}\overline{A}\]
as desired.

The assertion follows from {\bf Claim 3}.

(2) Now $n$ is a prime integer $\neq 3,5$. By Proposition \ref{xxpro2.3}(1),
$\{1,\ldots, n-1\}\subseteq \Phi_n$, so the hypothesis of part (1) holds
when taking $\Phi=\Phi_n$. 
By part (1), $0\in \overline{\Psi}^{[n]}_1$. 
Since $c_k=0$ in $\overline{A}$ for all $k\neq 0\mod n$, we have that
$c_k^N=0$ in $\overline{B}_1 \overline{A}$ for all $k$. 
This implies that $\GKdim (\overline{A}/\overline{B}_1 \overline{A})=0$. 
By Lemma \ref{xxlem6.5}(2), 
$\GKdim \overline{E}=0$. The assertion follows from Lemma \ref{xxlem6.2}.
\end{proof}

Proposition \ref{xxpro6.6}(2) is one of the initial steps in the 
proof of Theorem \ref{xxthm0.2} and Proposition \ref{xxpro6.6}(1)
is a step of reduction. The following technical lemma is needed for
the proof of the proposition below.

\begin{lemma}
\label{xxlem6.7} 
Suppose $i$ and $j$ satisfy the following conditions:
\begin{enumerate}
\item[(1)]
$i$ is an {\rm{(}}odd{\rm{)}} integer with $0\leq i\leq n-1$ 
and $\gcd(i,n)>1$,
\item[(2)]
$2\leq j\leq n-1$ such that $\gcd(i,j,n)=1$.
\end{enumerate}
Then there is an integer $t\geq 0$ such that $\gcd(j+ti,n)=1$.
\end{lemma}

\begin{proof}
Let $n=p_1^{n_1} \cdots p_{s}^{n_s}\cdots p_{r}^{n_r}$ for some
$1\leq s\leq r$, where $\{p_i\}$ are the prime factors of $n$ and
$n_u\geq 1$ for all $1\leq u\leq r$. The ordering of the prime and the
integer $s$ are chosen so that $j=p_1^{j_1} \cdots p_{s-1}^{j_{s-1}} j'$ 
where $\gcd(j',n)=1$ and $j_w\geq 1$ for all $1\leq w \leq s-1$. 
Since we assume that $\gcd(i,j,n)=1$, we can write 
$i=p_{s}^{i_{s}}\cdots p_{r}^{i_r} i'$ where $\gcd(i',n)=1$ 
with $i_v\geq 0$ for all $s\leq v\leq r$. Let $t=p_{s}\cdots 
p_{r}$. Then it is easy to see that each $p_u$, for 
$1\leq u\leq r$, does not divide $j+ ti$. Thus $\gcd(j+ti,n)=1$.
\end{proof}

\begin{proposition}
\label{xxpro6.8} 
Let $n\geq 2$ and denote $\Phi=\Phi_n$. Suppose that
\begin{enumerate}
\item[(a)]
every proper factor of $n$ is admissible, and that
\item[(b)]
for each $0\leq i\leq n-1$,
$i\in \overline{\Psi}^{[n]}_1$.
\end{enumerate} 
Then $n$ is admissible.
\end{proposition}

\begin{proof} By Lemmas \ref{xxlem6.2} and \ref{xxlem6.3}, 
and the ideas in Lemma \ref{xxlem3.4},
it suffices to show that $i\in \overline{\Psi}^{[n]}_j$
for all $0\leq i\leq n-1$ and $1\leq j\leq n-1$ with $j\mid n$. 
We use induction on $j$ and then on $i$. The minimal possible
$j$ is $1$, in which the assertion follows from hypothesis (b).
Now assume that $j>1$.

If $\gcd(i,n)=1$, or $i=2^w q$ with $q$ odd and $\gcd(q,n)=1$, then 
the assertion follows from Proposition \ref{xxpro2.3}(1). This shows
that the assertion holds for $i=1$. So we can assume that 
$i\geq 2$ and proceed with induction on $i$. 

Suppose $i$ is even and write $i=2i'$. If $n$ is even, then
$i'\in \Psi^{[n/2]}_j$ by hypothesis (a). By Lemma 
\ref{xxlem3.9}(3), $i\in \Psi^{[n]}_{2j}$. Consequently,
$i\in \overline{\Psi}^{[n]}_{2j}$. By the $\overline{A}$-version 
of Lemma \ref{xxlem5.5}(3), $i\in \overline{\Psi}^{[n]}_{j}$. 
If $n$ is odd, then by the induction hypothesis,
$i'\in \overline{\Psi}^{[n]}_{j}$. Applying the automorphism
$f_2$ in Definition \ref{xxdef3.3}, we obtain
that $i\in \overline{\Psi}^{[n]}_{2j}$. By the $\overline{A}$-version
of Lemma \ref{xxlem5.5}(3), we have $i\in \overline{\Psi}^{[n]}_{j}$. 

For the rest of proof we assume that $i$ is odd and $\gcd(i,n)>1$.
If $\gcd(i,j,n)=:q>1$, write $n=qn'$, $i=qi'$ and $j=qj'$.
By hypothesis (a), $i'\in \Psi^{[n']}_{j'}$. By Lemma \ref{xxlem3.9}(3),
$i\in \Psi^{[n]}_{j}$. Consequently, $i\in \overline{\Psi}^{[n]}_{j}$.
The other alternative is $\gcd(i,j,n)=1$.
By Lemma \ref{xxlem6.7}, there is a $t\geq 0$ such that 
$\gcd(j+ti,n)=1$. By hypothesis (b), $i'\in \overline{\Psi}^{[n]}_{1}$ 
for all $i'$. Let $\lambda$ be $j+ti$, which is invertible
in ${\mathbb Z}_n$ by Lemma \ref{xxlem6.7}, and let $f_{\lambda}$ 
be the ($\overline{A}$-version of the) algebra automorphism 
defined as in Definition \ref{xxdef3.3}. Pick $i'=i\lambda^{-1}$ 
in ${\mathbb Z}_n$. Then, after applying $f_{\lambda}$, 
$i'\in \overline{\Psi}^{[n]}_{1}$ becomes 
$i\in \overline{\Psi}^{[n]}_{j+ti}$. By the 
$\overline{A}$-version of Lemma \ref{xxlem5.5}(3), 
$i\in \overline{\Psi}^{[n]}_{j}$ for all $i$. Thus 
we have finished the inductive step and the whole proof.
\end{proof}

\section{Proof of Theorem \ref{xxthm0.2}}
\label{xxsec7}

The proof of Theorem \ref{xxthm0.2} follows the strategy
of Proposition \ref{xxpro6.6}. Let us recall the argument 
for showing that some power of $c_0$ is in 
$\overline{B}_1\overline{A}$. Given an element 
$c_0^sb_j\in\overline{B}_1\overline{A}$, depending on 
whether $j$ satisfies a certain congruence, we applied 
commutators to $c_0^sb_j$ to conclude that 
$c_0^{s+1}b_{j+1}$ or $c_0^{s+2}b_{j+2}$ is in 
$\overline{B}_1\overline{A}$. Starting with $b_1$ and 
continuing in this way, we eventually reach 
$c_0^{n-2}b_{n-1}\in\overline{B}_1\overline{A}$. 
Applying $[-,b_1]$ to this gives 
$c_0^{n-1}\in\overline{B}_1\overline{A}$.

In the more general situation of Theorem \ref{xxthm0.2}, 
we have to show that for each divisor $i$ of $n$, some 
power of $c_i$ is in $\overline{B}_{i_0}\overline{A}$ 
for all $i_0$. Given an element $c_i^sb_j\in
\overline{B}_{i_0}\overline{A}$ we apply certain 
commutators to $c_i^sb_j$ depending on congruences 
satisfied by $i,j,i_0$ to conclude that other elements 
of the form $c_i^{s'}b_{j'}$ are in 
$\overline{B}_{i_0}\overline{A}$ (see Lemmas 
\ref{xxlem7.4} and \ref{xxlem7.5}). These congruences 
are described in Definition \ref{xxdef7.1}. Then we 
show that there are indeed integers satisfying 
Definition \ref{xxdef7.1} (see Lemma \ref{xxlem7.2}) 
and that we eventually reach $c_0^{s'}b_{i-i_0}$ 
(see Lemma \ref{xxlem7.6}), so that applying 
$[-,b_{i_0}]$ gives what we want. The final induction 
steps needed for the proof of Theorem \ref{xxthm0.2} 
are given in Proposition \ref{xxthm7.7} and Corollary 
\ref{xxcor7.8}.

% For each $n\geq 2$, let
% %\begin{equation}
%
% %\mmp(n):={\text{the minimal prime factor of $n$}}.
% %\end{equation}
% \begin{equation}
% \notag
% \mop(n):={\text{the minimal odd prime factor of $n$}}.
% \end{equation}
% %\begin{equation}
% 
% %\np(n):={\text{the number of prime factors of $n$}}.
% %\end{equation}
% %\begin{equation}
% 
% %\nop(n):={\text{the number of odd prime factors of $n$}}.
% %\end{equation}
%
% Throughout this section we let $\Phi=\Phi_n$.
% For two integers $i_0$ and $i$, to show 
% that $i\in \overline{\Psi}^{[n]}_{i_0}$, it is equivalent to 
% show that some power of $c_{i}$ is in $\overline{B}_{i_0}\overline{A}$.
% Our proof uses the following idea:
% we assume that $b_j c_{i}^t \in \overline{B}_{i_0}\overline{A}$
% for some $j$ and $t$ (for example, we can start with $j=i_0$ and $t=0$), 
% then to show that $b_{j'}c_i^{s+1}\in \overline{B}_{i_0}\overline{A}$ 
% for some particular choices of $j'$. Hopefully, 
% we can show that $b_{i-i_0}c_i^t\in \overline{B}_{i_0}\overline{A}$, then
% by taking a commutator, we have
% $$c_i^{t+1}=[b_{i-i_0}c_i^t,b_{i_0}]\in\overline{B}_{i_0}\overline{A}$$
% as desired. 

To simplify notation, let
\begin{equation}
\notag
\Lambda_{i,i_0}:=\{ j \mid b_j c_{i}^t (=c_{i}^t b_j)
\in \overline{B}_{i_0}\overline{A},
\quad {\text{for some }}t\geq 0\}. 
\end{equation}
It is clear that $i_0\in \Lambda_{i,i_0}$. Write 
$$i_{\xi}=i_0+\xi(i_0-i)$$
and 
$$\bar{i}_{\xi}=(-\xi)(i_0-i)$$
for all integers $\xi$. Since $i_0$ will be a fixed 
integer in most of proofs below, we hope that the 
probability of serious confusion is not high. Let 
\begin{equation}
\notag
\Xi_{i,i_0}:=\{ \xi  \mid i_{\xi}\in \Lambda_{i,i_0} 
\quad {\text{and}}
\quad 0\leq \xi\leq \frac{1}{2}(\mop(n)-3)\}. 
\end{equation}
where for $n\ge 2$
\begin{equation}
\notag
\mop(n):={\text{the minimal odd prime factor of $n$}}.
\end{equation}
It is clear that $0\in \Xi_{i,i_0}$.

Let $\mathbb{Z}_n^{\times}$ be the invertible elements in 
$\mathbb{Z}_n$ and $S$ be the set of odd integers between $1$ 
and $n$ which are not in $\phi_2(n)$ \eqref{E0.6.1}. Then define
\begin{align}
\label{E7.0.1}\tag{E7.0.1}
\Omega_{2}(n):=
\bigcap_{s\in S} \left(\mathbb{Z}_n^{\times}+s\right)
\end{align}
It is not hard to show that $\Omega_2(n)\subseteq \phi_2(n)$.

\begin{definition}
\label{xxdef7.1}
Let $n\geq 2$ be an integer. An integer $i_0$ with $1\leq i_0\leq n-1$ 
is called {\bf $n$-special} if
\begin{enumerate}
\item[(1)]
$i_0\in \phi_2(n)$. (This is also a consequence of (2) below.) 
\end{enumerate}
For all odd integers $i\not\in \phi_2(n)$ with $1\leq i\leq n$, the 
following hold.
\begin{enumerate}
\item[(2)]
$i_0-i$ is invertible in ${\mathbb Z}_n$.
\end{enumerate}
Part (2) is just that $i_0\in \Omega_2(n)$.
Now fix any $i$ as in part (2). For every $1\leq j \leq n-1$, either
\begin{enumerate}
\item[(3)]
$2j+(i_0-i)\in \phi_2(n)$,
\end{enumerate}
or 
\begin{enumerate}
\item[(4)]
(If $2j+(i_0-i)\not\in \phi_2(n)$, then) there is a $\xi \in 
\Xi_{i,i_0}$ such that
\begin{enumerate}
\item[(4$(\xi)$i)]
$2j+(\xi+2)(i_0-i)\in \phi_2(n)$,
\item[(4$(\xi)$ii)]
$j-(\xi+1)(i_0-i)\in \phi_2(n)$, and
\item[(4$(\xi)$iii)]
$j+i_0+2(\xi+1)(i_0-i)\in \phi_2(n)$.
\end{enumerate}
For any $\xi \in \Xi_{i,i_0}$, conditions (4$(\xi)$i), 
(4$(\xi)$ii) and (4$(\xi)$iii) all together are denoted by 
(4$(\xi)$). 
\end{enumerate}
Let $Spl(n)$ denote the set of integers $i_0$ that are 
$n$-special.
\end{definition}

\begin{lemma}
\label{xxlem7.2}
Let $n=2^a p^b$ where $p$ is a prime $\geq 3$. 
\begin{enumerate}
\item[(1)]
$2\in \Omega_2(n)$.
\item[(2)]
If $p\geq 7$, then 
$i_0=2$ is $n$-special.
\end{enumerate}
\end{lemma}

\begin{proof} (1) For every odd integer $i\not\in \phi_2(n)$, 
we have $p\mid i$. Therefore $2$ and $p$ do not divide 
$i_0-i=2-i$, so Definition \ref{xxdef7.1}(2) holds, and 
the assertion follows.

(2) Now assume $p\geq 7$. 
Note that Definition \ref{xxdef7.1}(1) is obvious.
Definition \ref{xxdef7.1}(2) holds by part (1). For 
Definition \ref{xxdef7.1}(3,4), note that $0\in \Xi_{i,i_0}$.
Note that, for every $i$ given in Definition \ref{xxdef7.1}(2),
$i\not\in \phi_2(n)$. Hence $i$ is divisible by $p$. 
When Definition \ref{xxdef7.1}(3) fails, namely, $2j+(i_0-i)$ 
(or equivalently, $2j+i_0$) is divisible by $p$, then 
$j+1=\frac{1}{2}(2j+i_0)$ is divisible by $p$, that is, 
$j=-1\mod p$. By taking $\xi=0$, we have 
$$\begin{aligned}
2j+2(2-i)& =-2+4=2\neq 0\mod p,\\
j-(1)(2-i)&=-1-2=-3\neq 0\mod p,\\
j+2+2(1)(2-i)&=-1+2+4=5\neq 0\mod p.
\end{aligned}
$$
This means that (4$(0)$i), (4$(0)$ii) and (4$(0)$iii)
hold. Therefore $i_0=2$ is $n$-special.
\end{proof}

We have another case when $\Omega_2(n)$ is non-empty. The 
following lemma is not needed for the proof of Theorem 
\ref{xxthm0.2}. It will be used in \S\ref{xxsec8} (see 
Theorem \ref{xxthm8.7}).

\begin{lemma}
\label{xxlem7.3}
Let $p_1$ and $p_2$ be two distinct odd primes.
\begin{enumerate}
\item[(1)]
If $n=p_1p_2$, then 
$\Omega_2(n)\neq \emptyset$.
\item[(2)]
If $n$ is either $p_1 p_2^2 n'$ or $2 p_1 p_2 n'$ 
for some $n'\geq 1$, then 
$\Omega_2(n)= \emptyset$. As a consequence, $Spl(n)=\emptyset$.
\end{enumerate}
\end{lemma}

\begin{proof} (1) Let $p_1<p_2$. Every integer $m$ can 
be written uniquely as $m=a p_1+b p_2$ where $0\leq a< p_2$. 
In particular,
$1=a_1 p_1+b_1 p_2$. If $a_1$ is odd, we claim that $i_0=-1
\in \Omega_2(n)$. If $a_1$ is even, we claim that $i_0=1\in 
\Omega_2(n)$. Since the proofs are very similar, we only 
consider the first case. 

Suppose that $i$ is an odd 
integer $1\leq i\leq n$ that is not in $\phi_2(n)$ such 
that $i-i_0=i+1$ is not invertible in ${\mathbb Z}_n$. 
Then $i$ and $i+1$ are divisible by different 
prime factors. We need to consider two cases.

Case 1: $p_1\mid i$ and $p_2\mid i+1$. Write
$i=i' p_1$ (where $i'$ is odd as $i$ is odd) and 
$i+1=j p_2$. Then $$-1=i-(i+1)=i'p_1-j p_2$$
where $i'$ is odd, which implies that
$$1=(p_2-i')p_1+(j-p_1)p_2.$$
Note that $p_2-i'$ is even, which contradicts the fact
that $a_1$ is odd. 

Case 2: $p_2\mid i$ and $p_1\mid i+1$. Write
$i=i' p_2$ (where $i'$ is odd as $i$ is odd) and $i+1=j p_1$
(where $j$ is even as $i+1$ is even). Then
$$-1=i-(i+1)=i'p_2-j p_1,$$
which implies that
$$1=j p_1 -i'p_2.$$
Note that $j$ is even, which contradicts the fact
that $a_1$ is odd. 

(2) Since $p_1$ and $p_2$ are distinct, $1=a p_1+b p_2$.
For every $i_0$, one can write it as 
$$i_0= cp_1+ dp_2$$
for some $c,d$ with $0\leq c< p_2$. If $c$ is odd, 
take $i=cp_1<n$, which is odd and not in $\phi_2(n)$. 
Then $i_0-i=dp_2$ is not invertible in ${\mathbb Z}_n$.
If $c$ is even, take $i=(c+p_2)p_1<n$, which is odd
and not in $\phi_2(n)$. Then $i_0-i=(d-p_1) p_2$ 
is not invertible in ${\mathbb Z}_n$. This means
that $i_0\not\in \Omega_2(n)$ for every $i_0$.
\end{proof}

The next two lemmas describe a family of partially 
defined maps $\Lambda_{i,i_0} \dashto \Lambda_{i,i_0}$, 
where $i_0$ and $i$ satisfy hypotheses (1) and (2) of 
Definition \ref{xxdef7.1}.

\begin{lemma}
\label{xxlem7.4} 
Retain the above notation.
\begin{enumerate}
\item[(1)]
Let $j\in \Lambda_{i,i_0}$. If $j$ satisfies 
hypothesis (3) of Definition \ref{xxdef7.1}, 
that is, $2j+(i_0-i)\in \phi_2(n)$, then 
$j+(i_0-i)\in \Lambda_{i,i_0}$. In particular, 
we get a partially defined map $\omega_{0}: 
\Lambda_{i,i_0}\dashto \Lambda_{i,i_0}$
given by $\omega_{0}(j):= j+(i_0-i)$.
\item[(2)]
If $i_{\xi}\in \Lambda_{i,i_0}$ and $2i_{\xi}+(i_0-i)
=2i_0+(2\xi+1)(i_0-i) \in \phi_2(n)$, then 
$i_{\xi+1}\in \Lambda_{i,i_0}$.
\item[(3)]
If $2i_{\xi}+(i_0-i)=2i_0+(2\xi+1)(i_0-i) \in \phi_2(n)$
for all $0\leq \xi< \frac{1}{2}(\mop(n)-3)$,
then $\Xi_{i,i_0}=[0,1,\ldots,\frac{1}{2}(\mop(n)-3)]$.
\end{enumerate}
\end{lemma}

\begin{proof} Part (2) is a special case of part (1) by taking 
$j=i_{\xi}$. Part (3) follows from part (2) and induction. 
So we only prove part (1) below.

Let $s=i-j$ and $r=j-(i-i_0)$. Then $j+r=2j-i+i_0$ is in
$\phi_2(n)$ by the hypothesis. By Proposition \ref{xxpro2.3}(1),
$j+r\in \Phi_n$ and $b_j b_r=-b_rb_j$ in $\overline{A}$. We 
start with $b_j c_i^t\in \overline{B}_{i_0}\overline{A}$ for 
some $t\geq 0$ (as $j\in \Lambda_{i,i_0}$). By the choice of 
$r,s$, we have $b_r b_s=b_{j-i+i_0}b_{i-j}
\in \overline{B}_{i_0}\overline{A}$. 

Consider the commutator $[b_jc_i^t, b_r b_s]$, we have the 
following elements in $\overline{B}_{i_0}\overline{A}$
\begin{eqnarray*}
-[b_j c_{i}^{t}, b_r b_s] 
&=& c_{i}^t (-b_jb_rb_s + b_rb_sb_j)\\
&=& c_{i}^t( b_rb_jb_s + b_rb_sb_j)\\
&=& c_{i}^t b_r c_{j+s}\\
&=& b_r c_{i}^{t+1}=b_{j+(i_0-i)} c_i^{t+1}.\\
\end{eqnarray*}
The assertion follows. 
\end{proof}

\begin{lemma}
\label{xxlem7.5}
Let $\xi\in \Xi_{i,i_0}$ and $j\in \Lambda_{i,i_0}$. Suppose 
that $\xi$ and $j$ satisfy the hypotheses $(4(\xi))$ in 
Definition \ref{xxdef7.1}. Then $j+(\xi+2)(i_0-i)\in 
\Lambda_{i,i_0}$. In particular, we get a partially defined 
map $\omega_{\xi+1}: \Lambda_{i,i_0}\dashto \Lambda_{i,i_0}$ 
given by $\omega_{\xi+1}(j):= j+(\xi+2)(i_0-i)$.
% Suppose that $\xi\in \Xi_{i,i_0}$ such that
% \begin{enumerate}
% \item[(4$(\xi)$i)]
% $2j+(\xi+2)(i_0-i)\in \phi_2(n)$,
% \item[(4$(\xi)$ii)]
% $j-(\xi+1)(i_0-i)\in \phi_2(n)$, and
% \item[(4$(\xi)$iii)]
% $j+i_0+2(\xi+1)(i_0-i)\in \phi_2(n)$
% \end{enumerate}
% for some $j$. (We will usually apply this lemma with the 
%additional hypothesis that $2j+(i_0-i)\not\in \phi_2(n)$.)
%
% If $j\in \Lambda_{i,i_0}$, then 
% $j+(\xi+2)(i_0-i)\in \Lambda_{i,i_0}$. In this case we can partially
% define a map $\omega_{\xi+1}: \Lambda_{i,i_0}\to \Lambda_{i,i_0}$
% by sending $j$ to $\omega_{\xi+1}(j):= j+(\xi+2)(i_0-i)$.
\end{lemma}

\begin{proof} Note that $\xi\in \Xi_{i,i_0}$
means that $i_{\xi}\in \Lambda_{i,i_0}$ where
$i_{\xi}=i_0+\xi(i_0-i)$.

Let $a=i-j$, $c=j+(\xi+2)(i_0-i)$ and $d=-(\xi+1)(i_0-i)$. By 
hypotheses (4$(\xi)$i)-(4$(\xi)$iii), we have that $c+j$, $d+j$ 
and $i_{\xi}+c$ are in $\phi_2(n)$. This means that 
$[b_c,b_j]=0=[b_d,b_j]=[b_c,b_{i_\xi}]$ in $\overline{A}$. 

Starting with $b_j c_i^t\in \overline{B}_{i_0}\overline{A}$ for some
$t\geq 0$ (as $j\in \Lambda_{i,i_0}$), we have the two sets of 
equations in $\overline{B}_{i_0}\overline{A}$. The first set is
\begin{eqnarray*}
[b_j c_i^{t}, b_ab_cb_d] 
&=& c_i^t(b_jb_ab_cb_d + b_ab_cb_db_j) \\
&=& c_i^t(b_jb_ab_cb_d + b_ab_jb_cb_d)\\
&=& c_i^t c_{a+j}b_cb_d\\
&=& c_i^{t+1} b_cb_d.\\
\end{eqnarray*}
In the above computation, note that 
$b_ab_cb_d=b_{i-j}b_{j+(\xi+2)(i_0-i)}b_{-(\xi+1)(i_0-i)}\in 
\overline{B}_{i_0}$. 
We also need $b_j$ to skew commute with $b_d$ 
and $b_c$, and $a+j= i$. Since $\xi\in \Xi_{i,i_0}$, there
is a $t'\geq 0$ such that $b_{i_\xi} c_i^{t'}
\in \overline{B}_{i_0}\overline{A}$.
The second set of equations is
\begin{eqnarray*}
[c_i^{t+1}b_cb_d, b_{i_\xi} c_i^{t'}] 
&=& c_i^{t+1+t'}(b_cb_db_{i_\xi} -b_{i_\xi}b_cb_d) \\
&=& c_i^{t+t'+1}(b_cb_db_{i_\xi} +b_cb_{i_\xi}b_d) \\
&=& c_i^{t+t'+1} b_c c_{d+i_\xi} \\
&=& c_i^{t+t'+2} b_c =c_i^{t+t'+2} b_{j+(\xi+2)(i_0-i)}. \\
\end{eqnarray*}
Therefore $j+(\xi+2)(i_0-i)\in \Lambda_{i,i_0}$ and the 
assertion follows.
\end{proof}

We usually apply the above lemma with the additional hypothesis 
$2j+(i_0-i)\not\in \phi_2(n)$.

If $i_0$ is $n$-special and $i_0-i\in\mathbb{Z}_n^{\times}$, 
then by Definition \ref{xxdef7.1}, the unions of the domains 
of definition of  $\omega_0$ and $\omega_{\xi(i)+1}$ for 
$\xi\in \Xi_{i,i_0}$ is equal to $\Lambda_{i,i_0}$. In other 
words, for any $j\in \Lambda_{i,i_0}$ there is some $\omega_t$ 
which can be applied to $j$.

\begin{lemma}
\label{xxlem7.6} Suppose that $i_0$ is $n$-special. Then 
$-(i_0-i)\in \Lambda_{i,i_0}$. Consequently
$c_i^N\in \overline{B}_{i_0}\overline{A}$ for some $N\geq 0$.
\end{lemma}

\begin{proof} 
Suppose that $\bar{i}_{\xi}\in\Lambda_{i,i_0}$ for 
some $\xi \in [0,M]$, where $M=\frac{1}{2}(\mop(n)-3)$. 
Then $2\bar{i}_{\xi} +(i_0-i) = (1-2\xi) (i_0-i) 
\in\phi_2(n) $ since $2\xi-1$ is invertible if 
$0\le\xi\le M$. By Lemma \ref{xxlem7.4}(1) we 
have $\bar{i}_{\xi}$ is in the domain of definition 
of $\omega_0$, so $\omega_0(\bar{i}_{\xi}) = 
\bar{i}_{\xi-1} \in \Lambda_{i,i_0}$. Repeating 
the argument gives $\bar{i}_{-1}\in\Lambda_{i,i_0}$. 

Now, let $r$ be the maximal integer such that 
$\bar{i}_{\xi}\not\in\Lambda_{i,i_0}$ for every 
$\xi\in\{-1,0,\ldots,r\}$. In other words, 
$\bar{i}_{r+1}\in\Lambda_{i,i_0}$. Since $i_0$ 
is $n$-special, either Lemma \ref{xxlem7.4} or 
\ref{xxlem7.5} applies. That is, $\bar{i}_{r+1}$ 
is in the domain of definition of $\omega_{\zeta}$ 
for some $\zeta\in[0,M+1]$. Applying any such 
$\omega_{\zeta}$ gives 
$\omega_{\zeta}(\bar{i}_{r+1})
=\bar{i}_{r-\zeta}\in\Lambda_{i,i_0}$. By maximality 
of $r$, we have $r-\zeta < -1$ so $r < \zeta -1\le M$. 
Thus we may apply the argument in the first 
paragraph to conclude that $\bar{i}_{-1}\in\Lambda_{i,i_0}$. 

The above shows that we always have $\bar{i}_{-1}\in
\Lambda_{i,i_0}$ under the hypotheses that $i_0$ is 
$n$-special. Equivalently, $b_{i-i_0}c_i^t\in 
\overline{B}_{i_0}\overline{A}$. Finally, 
$c_i^{t+1}=[b_{i-i_0}c_i^t, b_{i_0}]\in 
\overline{B}_{i_0}\overline{A}$.
\end{proof}

Here is one of the main results of this section, 
which leads to Theorem \ref{xxthm0.2}.

\begin{theorem}
\label{xxthm7.7}
Let $n\geq 2$. Suppose that 
\begin{enumerate}
\item[(1)]
every proper factor of $n$ is admissible,
\item[(2)]
there is an $n$-special integer $i_0$.
\end{enumerate}
Then, for each $0\leq i\leq n-1$, $i\in 
\overline{\Psi}^{[n]}_1$.
\end{theorem}

\begin{proof} By Definition \ref{xxdef7.1}(1), we can 
express the $n$-special integer $i_0$ as the product 
$i_0=2^w g$ where $w\geq 0$ with $g$ odd and $\gcd(n,g)=1$. 
Since $g$ is invertible in ${\mathbb Z}_n$, by using 
the automorphism $f_{g}$ of $A$ defined in Definition 
\ref{xxdef3.3}, the assertion is equivalent to 
$i\in \overline{\Psi}^{[n]}_g$ for all $0\leq i\leq n-1$.
By the $\overline{A}$-version of Lemma \ref{xxlem5.5}(3),
it suffices to show the following claim.

\bigskip
\noindent
{\bf Claim:} 
for each $0\leq i\leq n-1$, $i\in \overline{\Psi}^{[n]}_{i_0}$.

\noindent
{\bf Proof of Claim:}
We prove the {\bf Claim} by induction on $i$ starting at 
$i=1$ and ending at $i=n$ (which is also $0$ in 
${\mathbb Z}_n$). We consider several cases.

\bigskip
\noindent
Case 1: $i=1$. 

The assertion follows from Proposition \ref{xxpro2.3}(1).
For the inductive step, we assume that $i\geq 2$ and 
that $i' \in \overline{\Psi}^{[n]}_{i_0}$ for all $i'<i$.

\bigskip
\noindent
Case 2: $i$ is even.

If $n$ is also even, it follows from Lemma \ref{xxlem5.5}(4)
that $i\in \Psi^{[n]}_{i_0}$. Passing to the quotient ring, 
we have $i\in \overline{\Psi}^{[n]}_{i_0}$ as desired.

If $n$ is not even, then $f_2$ in Definition \ref{xxdef3.3}
is an automorphism. Write $i=2i'$ for some $1\leq i'< i$.
By the induction hypothesis, $i'\in \overline{\Psi}^{[n]}_{i_0}$.
Applying $f_2$, we obtain that $i=2i'\in \overline{\Psi}^{[n]}_{2i_0}$.
By the $\overline{A}$-version of Lemma \ref{xxlem5.5}(3),
we have $i\in  \overline{\Psi}^{[n]}_{i_0}$. This takes care 
of the case when $i$ is even. For cases 3 and 4 below, we assume
that $i$ is odd.

\bigskip
\noindent
Case 3: $i$ is odd and $i\in \phi_2(n)$.

In this case, the assertion follows from Proposition \ref{xxpro2.3}(1)
as $c_i=0$ in $\overline{A}$. 

The remaining case to consider is 

\bigskip
\noindent
Case 4: $i$ is odd and $i\not\in \phi_2(n)$. By 
hypothesis, $i_0$ is $n$-special. 
By Lemma \ref{xxlem7.6}, $i\in  \overline{\Psi}^{[n]}_{i_0}$

Hence we finished the inductive step and therefore 
we complete the proof of the {\bf Claim}. 
\end{proof}

\begin{corollary}
\label{xxcor7.8}
Let $n\geq 2$. Suppose that 
\begin{enumerate}
\item[(1)]
every proper factor of $n$ is admissible,
\item[(2)]
$Spl(n)\neq \emptyset$, namely, there is 
an  $n$-special integer $i_0$.
\end{enumerate}
Then $n$ is admissible.
\end{corollary}

\begin{proof}
By Theorem \ref{xxthm7.7}, for each 
$0\leq i\leq n-1$, we have $i\in 
\overline{\Psi}^{[n]}_1$. The assertion then 
follows from Proposition \ref{xxpro6.8}.
\end{proof}

Now we are ready to show Theorem \ref{xxthm0.2}.

\begin{proof}[Proof of Theorem \ref{xxthm0.2}]
In this case $n=2^a p^b$ for some prime $p\geq 7$.
If $(a,b)=(0,1)$ or $(1,0)$, the assertion follows from 
Proposition \ref{xxpro6.6}(2). This takes care
of the initial step for induction.

By Lemma \ref{xxlem7.2}, $i_0=2$ is $n$-special, 
which is hypothesis (2) in Corollary \ref{xxcor7.8}.
Hypothesis (1) in Corollary \ref{xxcor7.8} follows
by induction. Hence we can conclude from 
Corollary \ref{xxcor7.8} that $n$ is admissible. 
By definition, $\GKdim (E )=0$. 
Hence $\p(A, G)=n$ and, by \cite[Theorem 3.10]{MU1}, 
$A^G$ is a graded isolated singularity.
\end{proof}

\begin{proof}[Proof of Corollary \ref{xxcor0.6}]
For each $n<77$, $n$ is either divisible by $3$ or $5$, 
or $n$ is of the form $2^a p^b$ for some prime $p\geq 7$.
Hence the assertion follows by Theorem \ref{xxthm0.2}
and \ref{xxthm0.4}.
\end{proof}

\begin{proof}[Proof of Corollary \ref{xxcor0.10}]
By Theorem \ref{xxthm0.2} and \cite[Theorem 5.7(1)]{BHZ1},
$A^G$ is a graded isolated singularity. 
By \cite[Theorem 1.5]{KKZ1}, $A^G$ is Gorenstein.

(1) This follows from \cite[Corollary 2.6]{MU1}.

(2) This follows from \cite[Theorem 3.15]{MU1}.

(3) This follows from \cite[Theorem 3.14]{MU1}.

(4) This follows from \cite[Theorem 1.3]{Ue}.
\end{proof}

Below is a slightly more general result than Theorem 
\ref{xxthm0.2}.

\begin{theorem}
\label{xxthm7.9}
Let $S$ be a set of integers $n\geq 2$. Suppose that
\begin{enumerate}
\item[(1)]
each $n$ in $S$ is not divisible by $3$ or $5$,
\item[(2)]
every proper factor of $n\in S$ is still in $S$.
\item[(3)]
for each $n\in S$, $Spl(n)\neq 
\emptyset$.
\end{enumerate}
Then every $n\in S$ is admissible.
\end{theorem}

\begin{proof} The assertion follows by induction on $n\in S$.
Since each $n$ in $S$ is not divisible by $3$ or $5$,
the initial step follows from Proposition \ref{xxpro6.6}(2).
Now we assume that the assertion holds for all proper
factors of $n$. The induction step follows from hypothesis
(3) and Corollary \ref{xxcor7.8}.
\end{proof}

\section{Partial results when $n=p_1p_2$}
\label{xxsec8}

In this section we give some partial answer to the case
when $n=p_1 p_2$ for $p_i$ being distinct primes.
Some lemmas works for the case when $n=2^a p_1^b p_2^c$. 

Another way of defining $\Omega_{2}(n)$ \eqref{E7.0.1} is 
the following 
\begin{align}
\notag
\Omega_{2}(n)&:=\\
&\notag
\{i_0\in {\mathbb Z}_n\mid {\text{if $1\leq i\leq n$ is odd and 
$i\not\in \phi_2(n)$, then $i_0-i\in {\mathbb Z}_n$ is invertible}}\}.
\end{align}
As noted in Section \ref{xxsec7}, $\Omega_2(n)\subseteq \phi_2(n)$.

In the rest of this section, {\bf let $n$ be $2^a p_1^b p_2^c$ 
where $p_1$ and $p_2$ are distinct odd primes $\geq 7$ and 
$b,c\geq 1$.} We start with a linear algebra fact.

\begin{lemma}
\label{xxlem8.1}
Let $i_0\in \phi_2(n)$ and $i\in [1,\ldots,n]$
such that $p_1\mid i$. If $p_2$ divides 
both $c_{11} i_0+c_{12} i$ and $c_{21}i_0+c_{22}i$,
then $p_2$ divides 
\begin{equation}
\notag
\Det:=\det \begin{pmatrix} c_{11} &c_{12}\\c_{21}&c_{22}\end{pmatrix}
=c_{11}c_{22}-c_{12}c_{21}.
\end{equation}
\end{lemma}

\begin{proof}
By easy linear algebra, $p_2$ divides both $\Det i_0$ and $\Det i$.
Since $p_2$ and $i_0$ are coprime, $p_2$ divides $\Det$.
\end{proof}

\begin{lemma}
\label{xxlem8.2}
Let $i_0\in \phi_2(n)$ and $i\in [1,\ldots,n]$ be an odd 
integer not in $\phi_2(n)$.
There is at most one integer 
\begin{equation}
\label{E8.2.1}\tag{E8.2.1}
\xi\in 
[0,1,\ldots, -1+\frac{1}{2}(\mop(n)-3)]
\end{equation}
such that $2 i_{\xi} +(i_0-i)\not\in \phi_2(n)$.
\end{lemma}

\begin{proof} Without loss of generality, we can assume
that $p_1$ divides $i$. For each $\xi$ in \eqref{E8.2.1},
$p_1$ does not divide $2 i_{\xi} +(i_0-i)=
(2\xi+3)i_0-(2\xi+1)i$, as 
$$2\xi+3< \mop(n):=\min\{p_1,p_2\}.$$
If there are $\xi_1$ and $\xi_2$ in \eqref{E8.2.1}
such that $2 i_{\xi_1} +(i_0-i)$ and 
$2 i_{\xi_2} +(i_0-i)$ are not in $\phi_2$, 
then $p_2$ must divide both $2 i_{\xi_1} +(i_0-i)$ and 
$2 i_{\xi_2} +(i_0-i)$ (or equivalently, 
divide both $(2\xi_1+3)i_0-(2\xi_1+1)i$ and 
$(2\xi_2+3)i_0-(2\xi_2+1)i$). By Lemma \ref{xxlem8.1},
$p_2$ divides $\Det$, and an easy computation shows that
$$\Det=\det \begin{pmatrix} 2\xi_1+3 &2\xi_1+1\\2\xi_2+3&2\xi_2+1\end{pmatrix}
=4(\xi_2-\xi_1).$$
By the choices of $\xi_1,\xi_2$ in \eqref{E8.2.1},
$p_2$ does not divide $4(\xi_2-\xi_1)$, a contradiction.
The assertion follows.
\end{proof}

\begin{lemma}
\label{xxlem8.3} 
Retain the hypothesis as in Lemma {\rm{\ref{xxlem8.2}}}.
Suppose that $\xi$ in \eqref{E8.2.1} is such that
$2i_{\xi}+(i_0-i)\not\in \phi_2(n)$, namely,
Definition {\rm{\ref{xxdef7.1}(3)}} fails for $j=i_{\xi}$. 
Then Definition {\rm{\ref{xxdef7.1}(4$(0)$)}} holds for 
$j=i_{\xi}$.
\end{lemma}

\begin{proof} By Lemma \ref{xxlem8.2}, 
$\xi$ is unique.
It suffices to show that the following elements 
in (4(0)i), (4(0)ii) and (4(0)iii) of Definition
\ref{xxdef7.1}(4), when $j=i_{\xi}$, are in $\phi_2$:
\begin{align}
2i_{\xi}+2(i_0-i)&=(2\xi+4)i_0-(2\xi+2)i=2((\xi+2)i_0-(\xi+1)i),
\label{E8.3.1}\tag{E8.3.1}\\
i_{\xi}-(i_0-i)&=\xi i_0-(\xi-1) i,\label{E8.3.2}\tag{E8.3.2}\\
i_{\xi}+i_0+2(i_0-i)&=(\xi+4)i_0-(\xi+2)i.\label{E8.3.3}\tag{E8.3.3}
\end{align}

Since $i\not\in\phi_2(n)$, either $p_1$ or $p_2$ divides $i$.
Without loss of generality, we say $p_1$ divides $i$.
By the proof of Lemma \ref{xxlem8.2}, we have $p_2$ divides 
$2i_{\xi}+(i_0-i)=(2\xi+3)i_0-(2\xi+1)i$. Since $\mop(n)\geq 7$,
all of $\xi+2, \xi, \xi+4$ are strictly less than $\mod(n)$.
Hence $p_1$ does not divide elements in \eqref{E8.3.1}-\eqref{E8.3.3}.
We claim that $p_2$ does not divide elements in 
\eqref{E8.3.1}-\eqref{E8.3.3}. If this is false, say, the element
in \eqref{E8.3.3} is divisible by $p_2$, then 
$p_2$ divides both $(2\xi+3)i_0-(2\xi+1)i$ and 
$(\xi+4)i_0-(\xi+2)i$. By Lemma \ref{xxlem8.1}, $p_2$
divides $\Det$, where
$$\Det=\det \begin{pmatrix} 2\xi+3& 2\xi+1\\\xi+4&\xi+2\end{pmatrix}
=-2(\xi-1).$$
However, by the choice of $\xi$ in \eqref{E8.2.1},
$p_2$ does not divide $\Det$, a contradiction. 
The claim is proved.

Finally, by the above, elements in \eqref{E8.3.1}-\eqref{E8.3.3}
are not divisible by either $p_1$ or $p_2$. Hence these 
elements are in $\phi_2(n)$. The assertion follows.
\end{proof}

\begin{lemma}
\label{xxlem8.4}
Retain the hypothesis as in Lemma {\rm{\ref{xxlem8.2}}}. Then 
either
$$\Xi_{i,i_0}=[0,1,2,\ldots,\frac{1}{2}(\mop(n)-3)]$$
or
$$\Xi_{i,i_0}=[0,1,2,\ldots,\xi,\widehat{\xi+1},\xi+2, 
\ldots,\frac{1}{2}(\mop(n)-3)].$$
The second case can happen only when there is 
a $\xi$ in \eqref{E8.2.1} such that $2i_{\xi}+(i_0-i)\not\in 
\phi_2(n)$.
\end{lemma}

By Lemma \ref{xxlem8.2}, $\xi$ is Lemma \ref{xxlem8.4} above
is unique if it exists.

\begin{proof}[Proof of Lemma \ref{xxlem8.4}]
If, for each $\xi$ in \eqref{E8.2.1}, $2i_{\xi}+(i_0-i)$ is
in $\phi_2(n)$, then, by Lemma \ref{xxlem7.4}(3),
$$\Xi_{i,i_0}=[0,1,2,\ldots,\frac{1}{2}(\mop(n)-3)]$$
which is the first case.

To prove the lemma, we may assume that 
$$\Xi_{i,i_0}\neq [0,1,2,\ldots,\frac{1}{2}(\mop(n)-3)]$$
and that there is a $\xi$ in \eqref{E8.2.1} such that 
$2i_{\xi}+(i_0-i)\not\in\phi_2(n)$. By Lemma 
\ref{xxlem8.2}, $\xi$ is unique, namely, 
for all $\xi'\neq \xi$ in \eqref{E8.2.1}, $2i_{\xi'}+(i_0-i)$ is
in $\phi_2(n)$. By Lemma \ref{xxlem7.4}(2) and induction,
$0,1,2,\ldots, \xi\in \Xi_{i,i_0}$. By Lemma \ref{xxlem8.3},
Definition {\rm{\ref{xxdef7.1}(4$(0)$)}} holds for 
$j=i_{\xi}$. By Lemma \ref{xxlem7.5} (for a different 
$\xi=0\in \Xi_{i,i_0}$ in the setting of Lemma \ref{xxlem7.5}), 
$$j+2(i_0-i)=i_{\xi}+2(i_0-i)=i_{\xi+2}$$
is in $\Lambda_{i,i_0}$. If $\xi+2\leq \frac{1}{2}(\mop(n)-3)$,
then $\xi+2\in \Xi_{i,i_0}$ by definition. By Lemma \ref{xxlem8.2},
$2i_{\xi'}+(i_0-i)\in \phi_{2}(n)$ for all
$$\xi+2\leq \xi'\leq -1+\frac{1}{2}(\mop(n)-3).$$
Using 
Lemma \ref{xxlem7.4}(2) and induction again,
$\xi+3,\ldots,\frac{1}{2}(\mop(n)-3)\in \Xi_{i,i_0}$. 
Therefore
$$\Xi_{i,i_0}=[0,1,2,\ldots,\xi,\widehat{\xi+1},
\xi+2, \ldots,\frac{1}{2}(\mop(n)-3)].$$
This finishes the proof.
\end{proof}

\begin{corollary}
\label{xxcor8.5}
If $\mop(n)\geq 11$, then 
$\Xi_{i,i_0}$ contains one of the following subsets
$$[0,1,2,3], \quad [0,2,3,4], \quad
[0,1,3,4], \quad [0,1,2,4].$$
\end{corollary}

\begin{proof} The assertion follows from Lemma 
\ref{xxlem8.4} and the fact $\frac{1}{2}(\mop(n)-3)\geq 4$.
\end{proof}

\begin{proposition}
\label{xxpro8.6} If $\mop(n)\geq 17$, then 
Definition {\rm{\ref{xxdef7.1}(4)}} holds 
automatically. As a consequence, 
$\Omega_2(n)=Spl(n)$.
\end{proposition}

\begin{proof} We start with the assumption that
$2j+(i_0-i)\not\in \phi_2(n)$. Without loss of
generality, we can assume that $p_1$ divides
$2j+(i_0-i)$. For simplicity, let $j_0=i_0-i$.
So $j_0$ is not divisible by either $p_1$ or $p_2$.
The assumption is that $2j+j_0$ is divisible by
$p_1$. By Lemma \ref{xxlem8.1} (replacing $(i,i_0)$
by $(j,j_0)$), elements of the forms in Definition 
\ref{xxdef7.1}(4$(\xi)$i) and (4$(\xi)$ii) for 
$\xi=0,1,2,3,4$ are not divisible by 
$p_1$ since the corresponding $\Det$ is not divisible
by $p_1\geq 17$. Further, any two distinct elements of the 
form in Definition \ref{xxdef7.1}(4$(\xi)$i) and (4$(\xi)$ii) 
can not be divided by $p_2$ either (using the fact
$p_2\geq 17$). This implies that there is only
one $\xi$, say $\xi_0\in [0,1,2,3,4]$, such that
either (4$(\xi_0)$i) or (4$(\xi_0)$ii) fails. Removing
$\xi_0$ from the list $\Xi_{i,i_0}$, we still have 
three integers $\{\xi_1,\xi_2,\xi_3\}\subseteq 
[0,1,2,3,4]\cap \Xi_{i,i_0}$ such that Definition 
\ref{xxdef7.1}(4$(\xi_s)$i) and (4$(\xi_s)$ii) 
hold for all $s=1,2,3$. It remains to show that 
Definition \ref{xxdef7.1}(4$(\xi_s)$iii) holds
for one of $s$. Suppose on the contrary that
Definition \ref{xxdef7.1}(4$(\xi_s)$iii) fails
for all three $s$. Then there are two $s$ such that
$j+i_0+2(\xi_s+1)(i_0-i)$ is divisible by the same 
prime factor, say $p_2$. Applying Lemma 
\ref{xxlem8.1} to these two element with $(j',j_0)=(j+i_0,i_0-i)$,
we obtain that $p_2$ divides $|\Det|=2|\xi_{s_1}-\xi_{s_2}|<\mop(n)$.
This is impossible. Therefore Definition 
\ref{xxdef7.1}(4$(\xi_s)$iii) holds for one of $s$. Thus
we show that Definition {\rm{\ref{xxdef7.1}(4)}} holds 
automatically. 

The consequence is clear.
\end{proof}

\begin{theorem}
\label{xxthm8.7}
Suppose $n=p_1 p_2$ where $p_s$ are prime $\geq 17$. Then 
$n$ is admissible. As a consequence $A^G$ has a graded isolated
singularity.
\end{theorem}

\begin{proof} Since
every proper factor of $n$ is admissible by Theorem
\ref{xxthm0.2}, hypothesis of Corollary \ref{xxcor7.8}(1)
holds. By Lemma \ref{xxlem7.3}(1),
$\Omega_2(n)\neq \emptyset$. By Proposition
\ref{xxpro8.6}, $Spl(n)\neq \emptyset$. 
Hence hypothesis of Corollary \ref{xxcor7.8}(2)
holds. The assertion now follows from Corollary \ref{xxcor7.8}.
\end{proof}

\section{Proof of Proposition \ref{xxpro0.8}}
\label{xxsec9}

We start with $n=6$ and $10$.

\begin{lemma}
\label{xxlem9.1}
Retain the notation as in Theorem \ref{xxthm0.4}.
If $n=6$, then $\p(A,C_n)=5$.
\end{lemma}

\begin{proof} First let $\Phi:=\Phi_{6}=\{1,2,4,5\}$.
By Lemma \ref{xxlem6.2}, $\GKdim E=\GKdim \overline{E}$.
It suffices to show that $\GKdim \overline{E}=1$. By 
Theorem \ref{xxthm0.4}, it is enough to show that 
$\GKdim \overline{E}\leq 1$. By Lemma \ref{xxlem6.5}(2),
$$\GKdim \overline{E}=
\max_{j}\GKdim (\overline{A}/\overline{B}_j\overline{A})$$
where $j$ ranges over $\{1,2,3\}$ (all positive integers
less than $6$ that divide $6$).

Case 1: $j=3$. Since $c_3\in \overline{B}_3\overline{A}$,
$c_3=0$ in $\overline{A}/\overline{B}_3\overline{A}$. 
By the definition of $\overline{A}$, $c_1=c_2=c_4=c_5=0$.
Therefore $c_i=0$ for $i=1,2,3,4,5$ in 
$\overline{A}/\overline{B}_3\overline{A}$. Therefore 
$\overline{A}/\overline{B}_3\overline{A}$ is a finitely generated
module over $\Bbbk[c_0]$, which implies that 
$\GKdim (\overline{A}/\overline{B}_3\overline{A})\leq 1$.

Case 2: $j=2$. We need to show that 
$\GKdim (\overline{A}/\overline{B}_2\overline{A})\leq 1$.
By \eqref{E1.0.3}, it is enough to show the claim that 
\begin{equation}
\label{E9.1.1}\tag{E9.1.1}
\GKdim (\overline{A}/(\overline{B}_2\overline{A}+c_3A))=0.
\end{equation}
Now we change $\Phi$ from $\{1,2,4,5\}$ to $\{1,2,3,4,5\}$. 
Re-cycle all notation such as $\overline{A}$, $\overline{B}_i$, etc,
for the new $\Phi$, claim \eqref{E9.1.1} becomes
\begin{equation}
\label{E9.1.2}\tag{E9.1.2}
\GKdim (\overline{A}/\overline{B}_2\overline{A})=0
\end{equation}
with $\Phi=\{1,2,3,4,5\}$. For the rest of the proof in Case 2, let
$\Phi=\{1,2,3,4,5\}$. Note that we have the following
elements in $\overline{B}_2\overline{A}$:
$$b_2, \quad b_3b_5, \quad b_1b_3b_4.$$
Taking commutators in $\overline{B}_2\overline{A}$, we have
the following computations in $\overline{B}_2\overline{A}$:
$$\begin{aligned}
\; [b_3b_5, b_1b_3b_4]&=
b_3b_5b_1b_3b_4+b_1b_3b_4b_3b_5\\
&=b_3^2(b_5b_1b_4-b_1b_4b_5)\\
&=\frac{1}{2} c_0^2 (b_5b_1+b_1b_5)b_4\\
&=\frac{1}{2} c_0^{3} b_4,\\
[c_0^{3} b_4, b_2]&=c_0^4.
\end{aligned}
$$
Therefore $c_0^4=0$ in 
$\overline{A}/\overline{B}_2\overline{A}$, and consequently,
$0\in \overline{\Psi}_2$. By definition,
$c_i=0$ in $\overline{A}/\overline{B}_2\overline{A}$
for all $i=1,2,3,4,5$. Therefore \eqref{E9.1.2} holds.

Case 3: $j=1$. We need to show that 
$\GKdim (\overline{A}/\overline{B}_1\overline{A})\leq 1$
for $\Phi=\{1,2,4,5\}$. Similar to the proof of Case 2,
it is sufficient to show 
\begin{equation}
\notag
\GKdim (\overline{A}/\overline{B}_1\overline{A})=0
\end{equation}
with new $\Phi=\{1,2,3,4,5\}$. But this is Proposition 
\ref{xxpro6.6}(1).

Combining these three cases, we finish the proof.
\end{proof}

\begin{lemma}
\label{xxlem9.2}
Retain the notation as in Theorem \ref{xxthm0.4}.
If $n=10$, then $\p(A,C_n)=9$.
\end{lemma}

\begin{proof} This proof is very similar to the proof
of Lemma \ref{xxlem9.1}

First we let $\Phi:=\Phi_{10}=\{1,2,3,4,6,7,8,9\}$.
By Lemma \ref{xxlem6.2}, $\GKdim E=\GKdim \overline{E}$.
It suffices to show that $\GKdim \overline{E}=1$.
By Lemma \ref{xxlem6.5}(2),
$$\GKdim \overline{E}=\max_{j}\GKdim (\overline{A}/\overline{B}_j\overline{A})$$
where $j$ ranges over $\{1,2,5\}$ (all positive integers
less than $6$ that divide $10$).

Case 1: $j=5$. The proof of Case 1 in Lemma \ref{xxlem9.1} can
be easily modified by replacing $j=3$ to $j=5$.

Case 2: $j=2$. We need to show that 
$\GKdim (\overline{A}/\overline{B}_2\overline{A})\leq 1$.
By \eqref{E1.0.3}, it is enough to show the claim that 
\begin{equation}
\label{E9.2.1}\tag{E9.2.1}
\GKdim (\overline{A}/(\overline{B}_2\overline{A}+c_5A))=0.
\end{equation}
Now we change $\Phi$ from $\{1,2,3,4,6,7,8,9\}$ to $\{1,2,3,4,5,6,7,8,9\}$. 
Recycle all notation such as $\overline{A}$, $\overline{B}_i$, etc,
for the new $\Phi$, claim \eqref{E9.2.1} becomes
\begin{equation}
\label{E9.2.2}\tag{E9.2.2}
\GKdim (\overline{A}/\overline{B}_2\overline{A})=0
\end{equation}
with new $\Phi=\{1,2,3,4,5,6,7,8,9\}$. For the rest of the proof in Case 2, 
we use this new $\Phi$. Note that we have the following
elements in $\overline{B}_2\overline{A}$:
$$b_2, \quad b_4b_8, \quad b_1b_5b_6, \quad b_5 b_8 b_9.$$
Taking commutators in $\overline{B}_2\overline{A}$, we have
the following computations in $\overline{B}_2\overline{A}$:
$$\begin{aligned}
\; [b_4 b_8, b_2]&= c_0 b_4,\\
[b_1 b_5 b_6, c_0 b_4]&=c_0^2 b_1 b_5,\\
[b_5 b_8 b_9, c_0^2 b_1 b_5]&=-\frac{1}{2}c_0^{4} b_8,\\
[b_2, c_0^{4} b_8]&=c_0^5.
\end{aligned}
$$
Therefore $c_0^5=0$ in 
$\overline{A}/\overline{B}_2\overline{A}$, and consequently,
$0\in \overline{\Psi}_2$. By definition,
$c_i=0$ in $\overline{A}/\overline{B}_2\overline{A}$
for all $i=1,2,3,4,5,6,7,8,9$. Therefore \eqref{E9.2.2} holds.

Case 3: $j=1$. The proof of Case 3 in Lemma \ref{xxlem9.1}  works.

Combining these three cases with Theorem \ref{xxthm0.4},
we finish the proof.
\end{proof}

Next we consider $n=9$.

\begin{lemma}
\label{xxlem9.3}
Retain the notation as in Theorem \ref{xxthm0.4}.
If $n=9$, then $\p(A,C_n)=8$.
\end{lemma}

\begin{proof} 
First we let $\Phi:=\Phi_{9}=\{1,2,4,5,7,8\}$.
By Lemma \ref{xxlem6.2}, $\GKdim E=\GKdim \overline{E}$.
It suffices to show that $\GKdim \overline{E}=1$.
By Lemma \ref{xxlem6.5}(2),
$$\GKdim \overline{E}=\max_{j}\GKdim (\overline{A}/\overline{B}_j\overline{A})$$
where $j$ ranges over $\{1,3\}$. So we need to consider two cases.

Case 1: $j=3$. Since $b_3, c_3\in \overline{B}_3\overline{A}$,
$c_6, c_3\in \overline{B}_3\overline{A}$. This shows that
$c_i=0$ in $\overline{A}/\overline{B}_3\overline{A}$ for all
$i=1,2,3,4,5,6,7,8$. So 
$\GKdim (\overline{A}/\overline{B}_3\overline{A})\leq 1$.

Case 2: $j=1$. We need to show that 
$\GKdim (\overline{A}/\overline{B}_1\overline{A})\leq 1$.
Note that we have the following
elements in $\overline{B}_1\overline{A}$:
$$b_1, \quad c_6b_5b_8, \quad c_6^2 b_7.$$
Taking commutators in $\overline{B}_1\overline{A}$, we have
the following computations inside $\overline{B}_1\overline{A}$:
$$\begin{aligned}
\; [b_1, c_6 b_5 b_8]&= b_1 c_6 b_5 b_8-c_6 b_5 b_8 b_1\\
&=c_6[b_1 b_5] b_8-c_6 b_5 b_8 b_1\\
&=c_6 [c_6 b_8- b_5 b_1 b_8]-c_6 b_5 b_8 b_1\\
&=c_6^2 b_8 -c_0 c_6 b_5,\\
[c_6^2 b_7, c_6^2 b_8 -c_0 c_6 b_5]
&=c_6^5 - c_0 c_3 c_6^3.
\end{aligned}
$$
Similarly we have the following
elements in $\overline{B}_1\overline{A}$:
$$c_3 b_7, \quad c_3b_5 b_8, \quad c_3^2 b_4.$$
Taking commutators in $\overline{B}_1\overline{A}$, we have
the following computations inside $\overline{B}_1\overline{A}$:
$$\begin{aligned}
\; [c_3 b_7, c_3^2 b_5 b_8]&= c_3^3(b_7 b_5 b_8-b_5 b_8 b_7)\\
&=c_3^3(c_3b_8 -b_5b_7b_8-b_5 b_8 b_7)\\
&=c_3^4 b_8-c_3^3 c_6 b_5,\\
[c_3^2 b_4, c_3^4 b_8-c_3^3 c_6 b_5]
&=c_3^8 -c_3^6 c_0 c_6.
\end{aligned}
$$
It is easy to see that the quotient algebra
$$D:=\frac{\overline{A}}{(c_6^5 - c_0 c_3 c_6^3,c_3^8 -c_3^6 c_0 c_6)}$$
has GKdimension 1. Since 
$\overline{A}/\overline{B}_1\overline{A}$ is a quotient 
of $D$ by the above computation. Therefore
$\GKdim\overline{A}/\overline{B}_1\overline{A}\leq 1$
as desired.

Combining these two cases with Theorem \ref{xxthm0.4},
we finish the proof.
\end{proof}

Now we are ready to prove Proposition \ref{xxpro0.8}.

\begin{proof}[Proof of Proposition \ref{xxpro0.8}]
When $n=6, 10, 9$, the $p$ is $5, 9, 8$ by 
Lemmas \ref{xxlem9.1}, \ref{xxlem9.2} and \ref{xxlem9.3}
respectively. 
For $n=3,5$, the assertion follows by 
Propositions \ref{xxpro4.1} and \ref{xxpro4.2}
respectively. For $n=2,4,7,8,11,13,14$, the 
assertion follows from Theorem \ref{xxthm0.2}.
The statement for $n=12$ follows by combining
Theorems \ref{xxthm0.4} and \ref{xxthm0.7}.
\end{proof}

\section{More examples of graded isolated singularities}
\label{xxsec10}
{\bf To save space, we will omit some non-essential 
details in Sections \ref{xxsec10} and \ref{xxsec11}.}

In this section, we give more examples of graded isolated 
singularities. Some nice results of He-Y.H. Zhang \cite{HZ}
and Gaddis-Kirkman-Moore-Won \cite{GKMW} will be reviewed and used 
in this section. First we recall some definitions from \cite{HZ}.

Let $R$ be a noetherian algebra and $G$ be a finite group acting 
on $R$. We say that two sequences $(a_1,\ldots, a_w)$ and 
$(b_1, \ldots, b_w)$ of elements of $R$ are {\it pertinent under 
the $G$-action}, if 
\begin{equation}
\notag
\sum_{i=1}^{w} a_i (g\cdot b_i)=0
\end{equation}
for all $1\neq g\in G$. In this case we write
$(a_1,\ldots, a_w)\sim (b_1, \ldots, b_w)$.
The {\it radical} of 
the $G$-action on $R$ is defined to be
\begin{equation}
\notag
{\mathfrak r}(R,G):=\left\{
\sum_{i=1}^w a_i b_i \in R
\mid (a_1,\ldots, a_w)\sim (b_1, \ldots, b_w)\right\}.
\end{equation}
By \cite[Section 1]{HZ}, 
${\mathfrak r}(R,G)$ is a 2-sided ideal of $R$. 

Let $e_0$ be the element $1\#(\frac{1}{|G|} \sum_{g\in G} g)$ in
$R\# G$. By the proof of \cite[Proposition 2.4]{HZ}, 
${\mathfrak r}(R,G)=R\cap (e_0)$. Therefore we have
\cite[(3.1.1)]{HZ},
\begin{equation}
\notag
\p(R,G)=\GKdim R-\GKdim R/{\mathfrak r}(R,G).
\end{equation}
If $R$ is noetherian and Artin-Schelter regular, 
then $R^G$ is a graded isolated singularity if and only
if $R/{\mathfrak r}(R,G)$ is finite dimensional over the 
base field $\Bbbk$.

As said in introduction, almost all graded isolated 
singularities studied in this paper are 
non-conventional in the following sense.

\begin{definition}
\label{xxdef10.1} 
Let $R$ be a noetherian Artin-Schelter regular algebra with 
graded maximal ideal $\fm:=A_{\geq 1}$. Let $G$ be a finite 
subgroup of $\Aut_{gr}(R)$ such that $R^G$ is a graded isolated 
singularity. We say the graded isolated singularity $R^G$ is 
{\it non-conventional} if there is an element $1\neq \sigma \in G$
such that at least one of the eigenvalues of $\sigma$ restricted 
to the $\Bbbk$-vector space $\fm/\fm^2$ is 1. Otherwise, we say 
$R^G$ is {\it conventional}.
\end{definition}

If $R$ is the commutative polynomial ring $\Bbbk[V]$, then 
every graded isolated singularity $R^G$ is conventional, see 
\cite[Corollary 3.11]{MU1}. A similar statement holds for 
skew polynomial rings. Let $\{p_{ij}\mid 1\leq i<j\leq n-1\}$ be 
a set of nonzero scalars in $\Bbbk^{\times}$.
The skew polynomial ring $\Bbbk_{p_{ij}}[x_0,x_1,\ldots,x_{n-1}]$
is generated by $\{x_0,\ldots, x_{n-1}\}$, with $\deg x_i>0$ for 
each $i$, and subject to the relations $x_j x_i=p_{ij} x_i x_j$ 
for all $i<j$. Let $V=\bigoplus_{i=0}^{n-1} \Bbbk x_i$. 

\begin{lemma}
\label{xxlem10.2}
Let $R$ be a skew polynomial ring $\Bbbk_{p_{ij}}[x_0,x_1,\ldots,x_{n-1}]$
and let $G$ be a finite group acting on $R$ linearly and diagonally, namely,
each $x_i$ is an eigenvector of $G$. Then 
$R^G$ is a graded isolated singularity if and only if the $G$-action on
$V\setminus \{0\}$ is free.
\end{lemma}

\begin{proof} Let $d=|G|$. 

$\Longleftarrow$: Assume that the $G$-action on $V\setminus \{0\}$ is free.  
In this setting, for each $i$, the $G$-action on $\Bbbk x_i\setminus \{0\}$ is also
free. This implies that there is an $\sigma\in G$ and a $\xi\in \Bbbk$ 
being a primitive $d$th root of unity such that $\sigma (x_i)=\xi x_i$. 
As a consequence, $G$ is generated by $\sigma$ and $\sigma^w (x_i)=\xi^{w} x_i$ 
for all $w\in {\mathbb Z}_d$. By \cite[Lemma 3.4]{HZ}, $x_i^d\in 
{\mathfrak r}(R,G)$. Therefore $R/{\mathfrak r}(R,G)$ is finite dimensional. 
As a consequence, $R^G$ is a graded isolated singularity.

$\Longrightarrow$: We prove the statement by contradiction and assume 
that the $G$-action on $V\setminus \{0\}$ is not free. Pick an element
$1\neq \sigma\in G$ so that $\sigma$ has a fixed point in 
$V\setminus \{0\}$. This implies that $\sigma$ fixes one $x_i$.
Replacing $G$ by the subgroup $\langle \sigma \rangle$, we can assume
that $G=\langle \sigma \rangle$ following \cite[Theorem 3.4]{GKMW}. 
Since $\sigma$ fixes $x_i$, one can show that $x_i^N$ is not in 
${\mathfrak r}(R,G)$ for all $N\geq 0$ (which also follows from Lemma 
\ref{xxlem10.4}(6) in an appropriate setting). 
Therefore $R/{\mathfrak r}(R,G)$ is not finite dimensional, 
whence $R^G$ is not a graded isolated singularity.
\end{proof}

As a consequence of \cite[Theorem 3.4]{GKMW}, if $R^G$ is a graded
isolated singularity, then so is $R^H$ for all subgroups
$1\subsetneq H \subseteq G$. The graded isolated singularities 
in the above lemma are all conventional. One nice example of 
non-conventional graded isolated singularities is given by 
Gaddis-Kirkman-Moore-Won \cite{GKMW}.

\begin{example} \cite[Theorem 5.2]{GKMW}
\label{xxex10.3}
Let $R$ be a generic 3-dimensional Sklyanin algebra $S(a, b, c)$ 
generated by $\{x,y,z\}$ with standard relations, see 
\cite[Introduction]{GKMW}. Let $G$ be the cyclic group of order 3 
acting on $R$ by permuting the standard generators $\{x,y,z\}$. 
Then $R^G$ is a graded isolated singularity by 
\cite[Theorem 5.2]{GKMW}. Since $G$ has a fixed point $x+y+z$ in 
$R_1\setminus \{0\}$, we obtain that $R^G$ is non-conventional.
\end{example}

We will use a few more lemmas. In Lemma \ref{xxlem10.4} below
we do not assume that the $G$-actions is inner-faithful.

\begin{lemma}
\label{xxlem10.4}
Let $R$ and $S$ be two connected graded algebra with $G$-action
where $G$ is a finite group. Let $e_0=1\# (\frac{1}{|G|}\sum_{g\in G} g)$. 
Suppose that $f: R\to S$ be a graded algebra homomorphism that 
is compatible with $G$-action.
\begin{enumerate}
\item[(1)]
There is an induced algebra homomorphism $f\#G :
R\# G\to S\# G$ such that $f\# G(r\# g)=f(r)\# g$
for all $r\in R$ and $g\in G$.
\item[(2)]
$f\# G$ maps $e_0\in R\# G$ to $e_0\in S\# G$. As a consequence,
there is an induced algebra homomorphism $\overline{f\#G} :
R\# G/(e_0)\to S\# G/(e_0)$.
\item[(3)]
If $x\in R$ such that $x:=x\#1\in (e_0)$ in $R\# G$, then 
$f(x):=f(x)\# 1 \in (e_0)$ in $S\# G$. 
\item[(4)]
If $f$ is surjective, so is $f\# G$. If, further, $S\# G/(e_0)$
is infinite dimensional, so is $R\# G/(e_0)$.
\item[(5)]
$f$ maps ${\mathfrak r}(R,G)$ to ${\mathfrak r}(S,G)$.
As a consequence, $f$ induces an algebra homomorphism 
from $R/{\mathfrak r}(R,G)$ to $S/{\mathfrak r}(S,G)$.
\item[(6)]
Suppose $f$ is surjective. If
$S/{\mathfrak r}(S,G)$ is infinite dimensional, then 
so is $R/{\mathfrak r}(R,G)$.
\end{enumerate}
\end{lemma}

The proof of Lemma \ref{xxlem10.4} is easy and omitted. 

\begin{lemma}
\label{xxlem10.5}
Let $A=\Bbbk_{-1}[\mathbf{x}]$ with  $n\geq 2$. 
\begin{enumerate}
\item[(1)]
Let $p$ be a prime number such that $p\neq 3,5$ and $p\leq n$.
Then there is a group $G\subseteq \Aut_{gr}(A)$ of order $p$ such that 
$A^G$ is a non-conventional graded isolated singularity.
\item[(2)]
Let $p=3^a 5^b$ for some $a,b\geq 0$. If $G$ is a 
subgroup of $\Aut_{gr}(A)$ of order $p$ such that $A^G$ is a 
graded isolated singularity, then it is conventional.
\end{enumerate}
\end{lemma}

\begin{proof} We omit the proof of part (2). For part (1),
we show give a proof when $p=2$.

We construct the group $G=\langle \sigma \rangle$ as follows.
If $n$ is even, let $\sigma\in \Aut_{gr}(A)$
be defined by
$$\sigma: x_i \to x_{n-1-i}$$
for all $i\in {\mathbb Z}_n$. If $n$ is odd, let 
$\sigma\in \Aut_{gr}(A)$ be defined by
$$\sigma: x_i \to x_{n-1-i}, \quad
{\text{and}}\quad x_{\frac{n-1}{2}}\to - x_{\frac{n-1}{2}}$$
for all $i\in {\mathbb Z}_n$ not equal to $\frac{n-1}{2}$. 
By \cite[Example 1.6(ii)]{HZ} and \cite[Lemma 3.4]{HZ},
$x_i^2\in {\mathfrak r}(A,G)$ for all $i$. (Some details are 
omitted.) Therefore $A/{\mathfrak r}(A,G)$ is finite
dimensional and $A^G$ is a graded isolated singularity.
Since $G$ preserves $x_0+x_{n-1}$, it is non-conventional.
\end{proof}

The next lemma is due 
to Jason Bell. We thank him for sharing his result with us.
We say an algebra $B$ is PI if it satisfies a polynomial identity.

\begin{lemma}[Jason Bell]
\label{xxlem10.6}
Let $B$ be a noetherian connected graded PI algebra generated in
degree 1. If every linear combination of homogenous elements of 
odd degrees is nilpotent, then $B$ is finite dimensional.
\end{lemma}

\begin{proof}
Suppose on the contrary that $B$ is infinite dimensional.
Let $W$ be the set of graded ideals $I$ of $B$ such that
$B/I$ is infinite dimensional. Since $B$ is noetherian,
there is a maximal element $J$ in $W$. Replacing $B$ by
$B/J$, we may assume that every nonzero ideal of $B$ has 
finite codimension. Since $B$ is graded, every 
minimal prime of $B$ is graded. As a consequence, 
the nilradical $N$ of $B$ is graded. Since $B$ is noetherian,
$B$ is infinite dimensional if and only if $B/N$ is infinite
dimensional. This implies that $N=0$. As a consequence,
a product of minimal prime ideals is zero. This in turn 
implies that one of minimal prime is zero, or $B$ is prime.

Since $B$ is PI, there is a nonzero central element in $B$.
We can further assume that this element, say $z$, is homogeneous
and a nonzerodivisor (or regular element). By the last paragraph, 
$B/(z)$ is finite dimensional. Then $\GKdim B=1$ by \eqref{E1.0.3}.

By Small-Warfield's theorem \cite{SW}, the center $Z(B)$ of
$B$ is a finitely generated graded algebra of GKdimension one 
and $B$ is a finite module over $Z(B)$. Note that every 
nonzero element in $Z(B)$ is regular. Hence 
$Z(B)$ is contained in the second Veronese subring of $B$ since 
all odd degree elements are nilpotent. 

Let $Q:=Q_{gr}(B)$ be the graded quotient ring of $B$. 
By a graded version of Posner's theorem, this is just the 
result of inverting the homogeneous nonzero central elements, 
all of which have even degree. The important point here is that 
every element of odd degree in $Q$ can be written in the 
form $az^{-1}$ with $a, z$ homogeneous and $a\in B$ of odd degree 
and $z\in Z(B)$ of even degree. Let $T$ be the (ungraded) total
quotient ring of $B$ (or of $Q$). Then $T$ can be 
embedded into a matrix algebra over a field $F$. With 
this embedding, we fix a trace map $tr$ (the usual matrix trace).
(With a bit more care one can even show that $T\cong M_n(F)$
where $F$ is the fraction field of $Z(B)$.) In particular, 
$tr(1)\neq 0$. 

As a general fact, since $B$ is generated in degree 1, $Q$
is strongly ${\mathbb Z}$-graded in the sense of \cite[A.I.3]{NoV}.
Let $Q_{odd}:=\bigoplus_{i\; {\text{is odd}}} Q_i$
and $Q_{even}:=\bigoplus_{i\; {\text{is even}}} Q_i$.
Then $Q=Q_{odd}\oplus Q_{even}$ is a strongly ${\mathbb Z}_2$-graded
algebra, namely, $Q_{odd}^2=Q_{even}$. 
By the last paragraph, every element $u$ in
$Q_{odd}$ is of the form $a z^{-1}$ where $a\in B$ is a 
linear combination of homogeneous elements of odd degrees 
and where $z\in Z(B)$ is of even degree.
Therefore $u$ is nilpotent by hypothesis. Let $u,v$ be any two
elements in $Q_{odd}$. Then $u,v, u+v$ are all in $Q_{odd}$; 
and consequently, all nilpotent. By \cite[Lemma 1]{MOR},
$tr(uv)=0$. Since $Q_{odd}^2=Q_{even}$, $tr(Q_{even})=0$. 
This contradicts $tr(1)\neq 0$.
\end{proof}

Now we consider twisted tensor products.
Let $\{B(i)\}_{i=1}^{w}$ be a family of connected graded algebras.
Then the tensor product 
$$\bigotimes {^{n}} B(i):=
B(1)\otimes B(2)\otimes \cdots \otimes B(n)$$ 
is a connected graded and ${\mathbb Z}^{\oplus n}$-graded algebra.
Let $u_i$ denote the $i$th unit element $(0,\ldots,0,1,0,\ldots,0)\in 
{\mathbb Z}^{\oplus n}$ where $1$ is in the $i$th position.
Let $\{p_{ij}\in \Bbbk^{\times}\mid 1\leq i<j\leq n\}$ be 
a set of nonzero scalar. Define $f_{u_i}$ to be the 
${\mathbb Z}^{\oplus n}$-graded algebra automorphism of
$\bigotimes^{n} B(i)$ determined by
$$f_{u_i}(1^{\otimes (j-1)}\otimes x_j \otimes 1^{\otimes (n-j)})
=1^{\otimes (j-1)}\otimes x_j \otimes 1^{\otimes (n-j)}$$
for all $i\geq j$ and $x_j\in B(j)$
and
$$f_{u_i}(1^{\otimes (j-1)}\otimes x_j \otimes 1^{\otimes (n-j)})
= p_{ij}^{-\deg x_j} 1^{\otimes (j-1)}\otimes x_j \otimes 1^{\otimes (n-j)}$$
for all $i< j$ and homogeneous elements $x_j\in B(j)$.
Then 
$$F:=\{f_{u_1^{a_1}\cdots u_{n}^{a_n}}:=
f_{u_1}^{a_1}\cdots f_{u_{n}}^{a_n}\mid {u_1^{a_1}\cdots u_{n}^{a_n}
\in {\mathbb Z}^{\oplus n}}\}$$ 
is an twisting system of $\bigotimes^{n} B(i)$
in the sense of \cite[Definition 2.1]{Zh1}. By 
\cite[Proposition and Definition 2.3]{Zh1}, one can define a
twisted algebra of $\bigotimes^{n} B(i)$ associated to the 
twisting system $F$. This twisted algebra is denoted by 
$\bigotimes^{n}_{\{p_{ij}\}} B(i)$. If $B(i)=\Bbbk[x]$ for 
all $i$, then $\bigotimes^{n}_{\{p_{ij}\}} B(i)$ is 
canonically isomorphic to skew polynomial ring 
$\Bbbk_{p_{ij}}[x_1,\ldots,x_n]$, see \cite[p.310]{Zh1}. 
Note that if $a=1^{\otimes (i-1)}\otimes x_i \otimes 1^{\otimes (n-i)}$
and $b=1^{\otimes (j-1)}\otimes x_j \otimes 1^{\otimes (n-j)}$
for two homogeneous elements $x_i\in B(i)$ and $x_j\in B(j)$
for $i<j$. Then one can check that 
$$ba=p_{ij}^{\deg x_i \deg x_j} \; ab.$$
Suppose each $B(i)$ is a noetherian PI Artin-Schelter regular 
algebra (and it is possible that the ``PI'' hypothesis can be 
weakened). One can easily check that 
$\bigotimes^{n}_{\{p_{ij}\}} B(i)$ is noetherian and 
Artin-Schelter regular. Further, $\bigotimes^{n}_{\{p_{ij}\}} B(i)$ 
has {\it enough normal elements} in the sense of \cite[p.392]{Zh2}.
By \cite[Theorem 1]{Zh2}, it is Auslander regular and 
Cohen-Macaulay.

Suppose $G$ is a finite group and $\phi_i: G\to \Aut_{gr}(B(i))$
is an injective map for each $i$. Then there is a unique 
extension of the $G$-action on $\bigotimes^{n}_{\{p_{ij}\}} B(i)$.

\begin{proposition}
\label{xxpro10.7}
Retain the above notation. 
Suppose $G$ is a finite group and $\phi_i: G\to \Aut_{gr}(B(i))$
is an injective map for each $i$. Let 
$B=\bigotimes^{n}_{\{p_{ij}\}} B(i)$.
\begin{enumerate}
\item[(1)]
$B^G$ is a graded isolated
singularity if and only if each $B(i)^G$ is a graded isolated
singularity.
\item[(2)]
Assume $B^G$ is a graded isolated
singularity. Then $B^G$ is 
conventional if and only if each $B(i)^G$ is conventional.
\end{enumerate}
\end{proposition}

\begin{proof} The proof follows from Lemma \ref{xxlem10.4}(5,6).
Details are omitted.
\end{proof}

Proposition \ref{xxpro10.7} provides a lot examples of graded isolated 
singularities.

Next let $B(i)=B$, for $i=1,\ldots,n$, be a noetherian PI 
Artin-Schelter regular algebra generated in degree 1.
Let $p_{ij}=-1$ for all $i<j$. We consider $(-1)$-twisted 
tensor product $\bigotimes^{n}_{\{-1\}} B$ and the permutation
automorphism $\sigma\in \Aut_{gr}(\bigotimes^{n}_{\{-1\}} B)$
determined by

\begin{equation}
\label{E10.7.1}\tag{E10.7.1}
\sigma: 1^{\otimes (j-1)}\otimes x_j \otimes 1^{\otimes (n-j)}
\mapsto 1^{\otimes j}\otimes x_j \otimes 1^{\otimes (n-j-1)},
\quad 
1^{\otimes (n-1)} \otimes x_n \mapsto x_n\otimes 1^{\otimes (n-1)}
\end{equation}
for all $x_j, x_n\in B$. 

\begin{proposition}
\label{xxpro10.8}
Retain the above notation.  Assume that $n\geq 2$ is 
admissible in the sense of Definition \ref{xxdef5.2}(2). 
Let $B$ be any noetherian PI Artin-Schelter regular algebra
generated in degree 1.
Let $G$ be the group $\langle \sigma \rangle$ where $\sigma$
is defined in \eqref{E10.7.1}. Then 
$(\bigotimes^{n}_{\{-1\}} B)^G$ is a non-conventional graded 
isolated singularity.
\end{proposition}

\begin{proof}
Let $S=\bigotimes^{n}_{\{-1\}} B$. It suffices to show that
$S/{\mathfrak r}(S,G)$ is finite dimensional.

Let $x\in B$ be a linear combination of homogeneous elements 
of odd degrees.
Let $x_i=1^{\otimes i} \otimes x \otimes 1^{\otimes (n-i-1)}
\in S$, for $i=0,\ldots,n-1$. Then the subalgebra generated 
by $\{x_0,x_1,\ldots,x_n\}$ is the $(-1)$-skew polynomial
ring $R:=\Bbbk_{-1}[{\mathbf x}]$. So the inclusion
$f: R\to S$ is compatible with the $G$-action. (Note that 
$f$ is not a graded algebra homomorphism.) Since
$n$ is admissible, the quotient $R/{\mathfrak r}(R,G)$ is finite
dimensional. Hence, for each $x_i$, we have $x_i^N\in 
{\mathfrak r}(R,G)$ for some $N\geq 0$. By Lemma 
\ref{xxlem10.4}(5), $x_i^N\in {\mathfrak r}(S,G)$. 
This is true for all $x$ that is a linear combination of 
homogeneous elements of odd degrees in $B$.
By Lemma \ref{xxlem10.6}, the image of the map
$$B\to B\otimes \Bbbk^{\otimes (n-1)}\subset
{\textstyle{\bigotimes^{n}_{\{-1\}}}} 
B(=S)\to S/{\mathfrak r}(S,G)$$
is finite dimensional. Say this image is $\overline{B}$.
By symmetry, $S/{\mathfrak r}(S,G)$ is a quotient
ring of $\bigotimes^n_{\{-1\}}\overline{B}$, which is finite
dimensional. Therefore $S/{\mathfrak r}(S,G)$ is finite
dimensional as desired. 
\end{proof}

Proposition \ref{xxpro10.8} also provides a lot examples of 
graded isolated singularities by varying $B$.

\section{Some questions and comments}
\label{xxsec11}

It is quite reasonable to adapt Ueyama's definition of a
graded isolated singularity \cite[Definition 2.2]{Ue},
at least in the connected graded case. By Remark 
\ref{xxrem0.3}(2), the straightforward generalization of the 
freeness criterion for commutative quotient isolated 
singularities \cite[Lemma 2.1]{MSt} fails badly in the 
noncommutative case. However the freeness of the 
$G$-action on $V\setminus \{0\}$ is one of the 
easiest and most effective criterions for isolated 
singularities. Therefore we ask

\begin{question}
\label{xxque11.1}
What is the analogue of the freeness criterion of 
isolated singularities in the (connected graded) 
noncommutative setting?
\end{question}

Let $R$ be a noetherian Artin-Schelter regular algebra 
and let $G$ be a finite subgroup of $\Aut_{gr}(R)$. By 
a result of Mori-Ueyama \cite[Theorem 3.10]{MU1} together 
with \cite{HZ}, the following are equivalent:
\begin{enumerate}
\item[(1)]
$R^G$ is a graded isolated singularity,
\item[(2)]
$R/{\mathfrak r}(R,G)$ is finite dimensional,
\item[(3)]
$R\# G/(e_0)$, where $e_0=1\# (\sum_{g\in G} g)$, is finite dimensional,
\item[(4)]
$\p(R,G)=\GKdim R$.
\end{enumerate}
Mori-Ueyama's criterion of graded isolated singularities
is quite convenient. On the other hand, it could be very difficult 
to verify (2), or (3), or to calculate the exact value of $\p(R,G)$. 

One of the key steps in the proof of Theorem \ref{xxthm0.2} is to 
show that the set $Spl(n)$ is non-empty. But we can not prove that
$Spl(n)\neq \emptyset$ is necessary. In particular, we do not 
have answers to the following questions.

\begin{question}
\label{xxquie11.2}
Let $n=p_1 p_2$ for two distinct odd primes $p_1,p_2$.
\begin{enumerate}
\item[(1)]
If $7\leq \mop(n)\leq 17$, is then $n$ admissible?
\item[(2)]
Is $Spl(77)\neq \emptyset$? 
\item[(3)]
If $Spl(77)=\emptyset$, is $77$ admissible? 
\end{enumerate}
\end{question}

Hypersurface isolated singularities have been studied 
extensively, and form a rich topic in algebraic geometry
\cite{Mi}. The noncommutative version of a hypersurface 
was defined in \cite[Definition 1.3(c)]{KKZ2}. 

In the commutative theory, every hypersurface isolated 
singularity produces a finite dimensional Milnor algebra 
(as well as the Tjurina algebra). It would be interesting 
to develop a similar theory for the noncommutative 
hypersurface isolated singularities. At this point, it is 
not clear to us what is the best way of defining the 
noncommutative Jacobian ideal, since there are no 
canonically defined partial derivatives in the 
noncommutative case. Here we will like to propose a 
definition of the Milnor algebra when the hypersurface 
singularity is defined by ``double twisted superpotentials''.

Let $V$ be a finite dimensional vector space 
$\bigoplus_{s=1}^{v} \Bbbk x_i$, or $\{x_s\}_{s=1}^v$ be 
a basis of $V$. Let $F$ be the free algebra $\Bbbk\langle
x_1,\ldots,x_v\rangle=\Bbbk \langle V\rangle$.
Let $\sigma$ denote an element in ${\text{GL}}(V)$.
We define two $\Bbbk$-linear maps from $F$ to $F$. The 
first one is $\phi$, which is determined by
$$\phi: x_{i_1}\otimes \cdots \otimes x_{i_{n-1}}
\otimes x_{i_n}\mapsto 
x_{i_n}\otimes x_{i_1}\otimes \cdots \otimes x_{i_{n-1}}$$
for all $x_{i_s}$ in the basis of $V$. The second one 
$\sigma\otimes 1$, where $\sigma \in {\text{GL}}(V)$, 
is determined by
$$\sigma\otimes 1: x_{i_1}\otimes \cdots \otimes 
x_{i_{n-1}}\otimes x_{i_n} \mapsto \sigma(x_{i_1})
\otimes \cdots \otimes x_{i_{n-1}}\otimes x_{i_n}.$$
Following \cite[Definition 1]{DV}, \cite[p.1502]{BSW},
\cite[Definitions 2.1.3 and 2.1.4]{Ka},
\cite[Definition 2.5]{MSm}
(and taking the quiver with one vertex and $v$ arrows), 
a {\it twisted  superpotential} in the free algebra 
$F$ is an element $w$ in $F$ such that
$$w=(\sigma\otimes 1) \phi(w)$$
for some $\sigma\in {\text{GL}}(V)$. (All papers 
\cite{DV, BSW, Ka, MSm} use slightly different notation, 
but one can easily figure out the discrepancies). For 
every $x_i$, we define a partial derivation $\partial_i$ 
as follows
$$\partial_i (x_{i_1}\otimes x_{i_2}\otimes \cdots \otimes x_{i_w})
=\begin{cases} x_{i_2}\otimes \cdots \otimes x_{i_w}& i_1=i\\
0& i\neq i_1.\end{cases}$$
(This definition of a partial derivative is slightly 
different from the ordinary partial derivative in 
calculus. Another possibility is the cyclic, or circular, 
derivative.) For every $w$, let $\partial(w)$ be the 
$\Bbbk$-linear span of $\{\partial_i(w)\}_{i=1}^v$. For 
an integer $N$, one can define $\partial^N(w)$ inductively 
by $\partial^N(w)=\partial(\partial^{N-1}(w))$. Given a 
twisted superpotential $w$ and an integer $N$, one can 
define {\it superpotential algebra} ${\mathcal D}(w,N)$ 
\cite[Definition 2.1.6]{Ka} (which is the same as the 
derivation-quotient algebra in the sense of \cite{DV, BSW, MSm}) 
to be
$${\mathcal D}(w,N):=F/(\partial^N(w)).$$
Dubois-Violette proved a very nice result \cite[Theorem 11]{DV}:
a Koszul (or higher Koszul) algebra is twisted Calabi-Yau 
if and only if it is isomorphic to a superpotential algebra 
for a unique-up-to-scalar-multiples twisted superpotential $w$. 

\begin{definition}
\label{xxdef11.3} Retain the above notation.
\begin{enumerate}
\item[(1)]
A pair of elements $(w_1, w_2)$ in $F$ are called 
{\it double twisted superpotentials} if 
\begin{enumerate}
\item[(a)]
$w_1$ is a twisted superpotential (with an automorphism 
$\sigma_1\in {\text{GL}}(V)$) such that the superpotential 
algebra $D:={\mathcal D}(w_1,N)$ is a noetherian 
Artin-Schelter regular algebra. 
\item[(b)]
$w_2$ is a twisted superpotential (with an automorphism 
$\sigma_2\in {\text{GL}}(V)$) such that $w_2$ is a normal 
regular element in $D$. 
\end{enumerate}
Let $(w_1,w_2)$ be double twisted superpotentials in 
parts (2,3,4).
\item[(2)]
The algebra $D/(w_2)$ is called the hypersurface 
singularity associated to $(w_1,w_2)$, and is denoted 
by $T(w_1,w_2)$.
\item[(3)]
The {\it Milnor algebra} associated to $(w_1,w_2)$
is defined to be 
$${\mathcal M}(w_1,w_2):=D/(\partial(w_2)).$$
\item[(4)]
The {\it Milnor number} associated to  $(w_1,w_2)$
is defined to be
$$m(w_1,w_2):=\dim_{\Bbbk} {\mathcal M}(w_1,w_2).$$
\end{enumerate}
\end{definition}

With these definitions, we can ask the following:

\begin{question}
\label{xxque11.4}
Is $T(w_1,w_2)$ being a graded isolated singularity 
equivalent to $m(w_1,w_2)$
being finite?
\end{question}

The following example of a hypersurface isolated 
singularity is non-conventional such that Question 
\ref{xxque11.4} has an affirmative answer.

\begin{example}
\label{xxex11.5}
Let $A=\Bbbk_{-1}[x_0,x_1]$ and $G$ be the group
of automorphism of $A$ generated by $f$, where 
$f$ is determined by 
$$f: x_0\mapsto x_1, \quad x_1\mapsto x_0.$$
By \cite[Example 3.1]{KKZ1}, 
$A^G$ is a hypersurface singularity, which can be 
written as
$$A^G=D/(w_2)$$
where $D$ is an Artin-Schelter regular algebra of 
global dimension three and $w_2$ is a normal element 
of degree 6 in $D$. In details, $x=x_0+x_1$ and 
$y=x_0^3+x_1^3$,
$$D=\Bbbk\langle x,y\rangle/(x^2y-yx^2, xy^2-y^2x)$$ 
and 
$$w_2= 2x^6 -\frac{3}{2}(x^3y+x^2yx+xyx^2+yx^3)+4y^2.$$
By Theorem \ref{xxthm0.2}, $A^G$ has a non-conventional 
graded isolated singularity. 

Note that $D$ is $(-1)$-twisted Calabi-Yau 
\cite[Example 1.6]{RRZ}. There is a twisted superpotential 
$$w_1=xy^2 x +yx^2 y-y^2 x^2-x^2 y^2$$
with automorphism $\sigma$ determined by
$$\sigma: x\mapsto -x, y\mapsto -y,$$ 
and $D$ is the superpotential algebra associated to
$w_1$. It is easy to check that 
\begin{enumerate}
\item[(1)]
$w_2$ is a regular normal element in $D$,
\item[(2)]
$w_2$ is a superpotential.
\end{enumerate}
The Milnor ring of the hypersurface singularity 
$A^G$ is 
$$D/(\partial w_2)=D/(12x^5-\frac{3}{2}(x^2y+xyx+yx^2), 
-\frac{3}{2}x^3+4y),$$
which is isomorphic to $\Bbbk[x]/(x^5)$ by an easy 
calculation. As a consequence, the Milnor number of 
$A^G$ is 5. 

Note that the McKay quiver corresponding to $(A,G)$ 
is of type $\widetilde{L}_1$, see 
\cite[Proposition 7.1 and pp. 249-250]{CKWZ1}. This is 
slightly different from the classical $\widetilde{A}$, 
$\widetilde{D}$, $\widetilde{E}$ types.
\end{example}

\begin{remark}
\label{xxrem11.6}
Some other noncommutative hypersurface graded isolated 
singularities are given in \cite[Theorem 5.2]{CKWZ2} and 
\cite[Table 3 in p.537]{CKWZ2}. These are related to 
noncommutative McKay correspondence in dimension two. It 
would be interesting to answer Question \ref{xxque11.4} 
for these hypersurface singularities.
\end{remark}

\end{document}